\documentclass[11pt]{article}
\def\thetitle{Combinatorial sufficient conditions for graph rigidity and applications to random graphs}

\usepackage{graphicx}

\usepackage{amsmath,amssymb}

\usepackage[usenames,dvipsnames,svgnames,table]{xcolor}
\definecolor{CombinatoricaAqua}{HTML}{00698C}
\definecolor{CombinatoricaBlue}{HTML}{3A3293}
\definecolor{CombinatoricaBrown}{HTML}{66220C}
\definecolor{CombinatoricaRed}{HTML}{DF2A27}
\definecolor{HarvardCrimson}{rgb}{0.6471, 0.1098, 0.1882}
\definecolor{DAGreen}{HTML}{339900}

\makeatletter
\let\reftagform@=\tagform@
\def\tagform@#1{\maketag@@@
	{(\ignorespaces\textcolor{CombinatoricaBrown}{#1}\unskip\@@italiccorr)}}
\renewcommand{\eqref}[1]{\textup{\reftagform@{\ref{#1}}}}
\makeatother

\usepackage[backref=page]{hyperref}
\hypersetup{%
	unicode,
	pdfencoding=auto,
	pdfauthor={Peleg Michaeli},
	pdftitle={\thetitle},
	pdfsubject={},
	pdfkeywords={},
	colorlinks=true,
	citecolor=CombinatoricaBlue,
	linkcolor=CombinatoricaAqua,
	anchorcolor=CombinatoricaBrown,
	urlcolor=HarvardCrimson}

\usepackage{amsthm}
\usepackage{thmtools}
\usepackage{bbm}
\usepackage{enumitem}
\usepackage[capitalize]{cleveref}
\Crefname{mainthm}{Theorem}{Theorems}
\Crefname{fact}{Fact}{Facts}
\Crefname{claim}{Claim}{Claims}


\declaretheoremstyle[
spaceabove=\topsep, spacebelow=\topsep,
headfont=\color{CombinatoricaBrown}\normalfont\bfseries,
bodyfont=\itshape,
]{thm}
\declaretheoremstyle[
spaceabove=\topsep, spacebelow=\topsep,
headfont=\color{CombinatoricaBrown}\normalfont\bfseries,
bodyfont=\normalfont,
]{dfn}
\declaretheoremstyle[
spaceabove=0.5\topsep, spacebelow=0.5\topsep,
headfont=\color{CombinatoricaBrown}\normalfont\bfseries,
bodyfont=\normalfont,
]{rmk}

\declaretheorem[style=thm,parent=section]{theorem}
\declaretheorem[style=thm,sibling=theorem]{lemma}
\declaretheorem[style=thm,sibling=theorem]{corollary}
\declaretheorem[style=thm,sibling=theorem]{claim}
\declaretheorem[style=thm,sibling=theorem]{proposition}

\declaretheorem[style=rmk,numbered=no]{remark}

\declaretheorem[style=definition,sibling=theorem]{definition}

\usepackage[nobysame,msc-links,non-sorted-cites]{amsrefs}

\renewcommand{\eprint}[1]{\href{https://arxiv.org/abs/#1}{arXiv:#1}}

\BibSpec{book}{%
	+{}  {\PrintPrimary}                {transition}
	+{,} { \textbf}                     {title} 
	+{.} { }                            {part}
	+{:} { \textit}                     {subtitle}
	+{,} { \PrintEdition}               {edition}
	+{}  { \PrintEditorsB}              {editor}
	+{,} { \PrintTranslatorsC}          {translator}
	+{,} { \PrintContributions}         {contribution}
	+{,} { }                            {series}
	+{,} { \voltext}                    {volume}
	+{,} { }                            {publisher}
	+{,} { }                            {organization}
	+{,} { }                            {address}
	+{,} { \PrintDateB}                 {date}
	+{,} { }                            {status}
	+{}  { \parenthesize}               {language}
	+{}  { \PrintTranslation}           {translation}
	+{;} { \PrintReprint}               {reprint}
	+{.} { }                            {note}
	+{.} {}                             {transition}
	+{}  {\SentenceSpace \PrintReviews} {review}
}
\BibSpec{incollection}{%
  +{}  {\PrintAuthors}                {author}
  +{,} { \textit}                     {title}
  +{.} { }                            {part}
  +{:} { \textit}                     {subtitle}
  +{,} { \PrintContributions}         {contribution}
  +{,} { \PrintConference}            {conference}
  +{}  {\PrintBook}                   {book}
  +{,} { }                            {booktitle}
	+{}  { \PrintEditorsB}              {editor}
	+{,} { }                            {publisher}
  +{,} { \PrintDateB}                 {date}
  +{,} { pp.~}                        {pages}
  +{,} { }                            {status}
  +{,} { \PrintDOI}                   {doi}
  +{,} { available at \eprint}        {eprint}
  +{}  { \parenthesize}               {language}
  +{}  { \PrintTranslation}           {translation}
  +{;} { \PrintReprint}               {reprint}
  +{.} { }                            {note}
  +{.} {}                             {transition}
  +{}  {\SentenceSpace \PrintReviews} {review}
}

\makeatletter
\renewcommand{\PrintNames@a}[4]{%
	\PrintSeries{\name}
	{#1}
	{}{ and \set@othername}
	{,}{ \set@othername}
	{}{ and \set@othername}
	{#2}{#4}{#3}%
}
\makeatother

\makeatletter
\def\mathcolor#1#{\@mathcolor{#1}}
\def\@mathcolor#1#2#3{%
	\protect\leavevmode
	\begingroup
	\color#1{#2}#3%
	\endgroup
}
\makeatother
\definecolor{Red}{rgb}{0.618,0,0}
\definecolor{Blue}{rgb}{0,0,1}
\definecolor{Green}{rgb}{0,0.298,0}

\newcommand{\iso}{\mathrm{i}}

\usepackage{sectsty}
\sectionfont{\color{CombinatoricaBrown}}
\subsectionfont{\color{CombinatoricaBrown}}
\subsubsectionfont{\color{CombinatoricaBrown}}

\usepackage{soul}
\soulregister\ref7

\usepackage{pifont}
\usepackage{calc}

\usepackage[a4paper]{geometry}
\title{\thetitle}

\author{
  Michael Krivelevich\thanks{
    School of Mathematical Sciences,
    Tel Aviv University,
    Tel Aviv 6997801, Israel.
    Email: \href{mailto:krivelev@tauex.tau.ac.il}
                {\tt krivelev@tauex.tau.ac.il}.
    Research supported in part by NSF-BSF grant 2023688.
  }
  \and
  Alan Lew\thanks{
    Faculty of Mathematics, Technion Israel Institute of Technology, Technion City, Haifa 3200003, Israel;
    Email: \href{mailto:alanlew@technion.ac.il}
                {\tt alanlew@technion.ac.il}.
  }
  \and
  Peleg Michaeli\thanks{
    Mathematical Institute,
    University of Oxford,
    Oxford, UK.
    Email: \href{mailto:peleg.michaeli@maths.ox.ac.uk}
                {\tt peleg.michaeli@maths.ox.ac.uk}.
    Research supported by ERC Advanced Grant 883810.
    For the purpose of Open Access, the author has applied a CC BY public
    copyright licence to any Author Accepted Manuscript version arising from
    this submission.
  }
}

\makeatletter
\def\namedlabel#1#2{\begingroup
  #2%
  \def\@currentlabel{#2}%
  \phantomsection\label{#1}\endgroup
}
\makeatother

\usepackage{mleftright}
\mleftright
\newcommand{\defn}[1]{{\bfseries #1}}

\newcommand{\eps}{\varepsilon}
\renewcommand{\phi}{\varphi}

\newcommand{\RR}{\mathbb{R}}

\newcommand{\cM}{\mathcal{M}}

\newcommand{\sm}{\smallsetminus}
\newcommand{\es}{\varnothing}

\newcommand{\floor}[1]{\left\lfloor{#1}\right\rfloor}
\newcommand{\ceil}[1]{\left\lceil{#1}\right\rceil}

\newcommand{\vect}{\mathbf}

\newcommand{\pr}[0]{\mathbb{P}}
\newcommand{\E}[0]{\mathbb{E}}

\newcommand{\whp}[0]{\textbf{whp}}

\newcommand{\Dist}[1]{\mathsf{#1}}
\newcommand{\Bin}{\Dist{Bin}}

\newcommand{\p}{\vect{p}} 
\newcommand{\bq}{\vect{q}} 

\usepackage{tabto}

\usepackage{comment}

\usepackage{subcaption}
\usepackage{tikz,pgfplots}
\pgfplotsset{compat=1.16}
\usetikzlibrary{calc}

\begin{document}
\date{}
\maketitle

\begin{abstract}
  A graph $G=(V,E)$ is called $d$-rigid if, for a generic embedding of its vertices in $\RR^d$, every edge-length preserving continuous motion of the vertices preserves the distances between all pairs of non-adjacent vertices as well. In this paper, we present several new results on the rigidity of random graphs. In particular, we show that there exists $c>0$ such that, for $p\ge 2 \log{n}/n$, the binomial random graph $G(n,p)$ is with high probability (whp) $\lfloor c n p\rfloor$-rigid. This is sharp up to the constant $c$, and complements recent results of Peled and Peleg (in the regime $p= o(n^{-1/2})$), and of Jord\'an, Liu, and Vill\'anyi  (in the constant $p$ regime).
Moreover, we show that for every fixed $d\ge 2$ and $r\ge 501d$, a random $r$-regular graph is whp $d$-rigid, and that for $100/n\le p\le 2\log{n}/n$, the binomial random graph $G(n,p)$ contains whp an $\lfloor np/251\rfloor$-rigid subgraph with at least $(1-e^{-np/2})n$ vertices. Both results are sharp up to the multiplicative constant. In addition, we present a new sufficient condition for rigidity in terms of the minimum codegree of the graph (the minimum number of common neighbours of a pair of vertices in the graph).

A main tool in our arguments is a new combinatorial sufficient condition for rigidity, which provides a common generalization to Whiteley's vertex-splitting lemmas, and to the ``rigid partitions'' method, developed in works by Crapo, Lindemann, Lew, Nevo, Peled and Raz, and by the present authors.
\end{abstract}

\section{Introduction}

A \emph{$d$-dimensional framework} is a pair $(G,\p)$, where $G$ is a finite, simple graph, and $\p$ is a map from $V$ to $\RR^d$. We say that $(G,\p)$ is \emph{rigid} if  every edge-length preserving
 continuous motion of the vertices of $G$, starting from the positions prescribed by $\p$, preserves in fact the distances between all pairs of vertices (or, equivalently, if every such continuous edge-length preserving motion extends to an isometry of $\RR^d$).
We say that $\p:V\to\RR^d$ is \emph{generic} if the $d|V|$ coordinates of $\p(V)$ are algebraically independent over the rationals. A graph $G$ is called \emph{$d$-rigid} if $(G,\p)$ is rigid for some generic $\p:V\to \RR^d$ (or, equivalently, if $(G,\p)$ is rigid for \emph{all} generic $\p:V\to \RR^d$~\cite{AR78}).

It is easy to see that a graph is $1$-rigid if and only if it is connected.
For $d=2$, combinatorial characterizations of $2$-rigid graphs are known~\cites{PollaczekGeiringerberDG, laman1970graphs, LY1982generic, rescki1984network, crapo1990plane}.
However, in the case $d\ge 3$, the existence of a ``reasonable'' combinatorial characterization of $d$-rigidity remains a major open problem in the field. Given the lack of such a combinatorial characterization in higher dimensions, the search for combinatorial properties implying $d$-rigidity is of considerable interest.
For example, Vill\'anyi~\cite{Vil23+} recently established that every $d(d+1)$-vertex-connected graph is $d$-rigid, solving a long-standing conjecture of Lov\'asz and Yemini~\cite{LY1982generic} (who proved the $d=2$ case), and providing a natural counterpart to the (easy) fact that every $d$-rigid graph is $d$-vertex-connected; minimum degree conditions for rigidity were studied in~\cites{KLM25,krivelevich2024minimum, jordan2025degree}.

The study of the rigidity of random graphs has attracted increasing attention in recent years~\cites{jackson2007rigidity, theran2008rigid, kasiviswanathan2011rigidity, kiraly2013coherence, JT2022rigidity, LNPR23, KLM25, peled2024rigidity, jordan2025degree}. In addition to being a natural high-dimensional extension of classical work on connectivity properties of random graphs, research on the rigidity of random graphs provides fertile ground for the development and testing of new tools and methods in rigidity theory.

In this paper, we develop new combinatorial methods for determining the rigidity of graphs, and apply them to obtain new, close to sharp results on the rigidity of random graphs in different density regimes. In addition, we present a new sufficient condition for rigidity in terms of the minimum codegree of the graph, that is, the minimum number of common neighbours of a pair of vertices in the graph. Our main results are presented next.

Let $G(n,p)$ be the distribution over graphs on vertex set $[n]=\{1,\ldots,n\}$ in which every pair of vertices is connected by an edge independently with probability $p$. Recall that for a sequence $\{A_n\}_{n=1}^{\infty}$ of events in a probability space, we say that $A_n$ occurs \emph{with high probability}, or \whp{} for short, if $\pr(A_n)\to 1$ as $n\to \infty$.

In~\cite{LNPR23}, Lew, Nevo, Peled, and Raz proved that for fixed $d\ge 2$, the threshold for $d$-rigidity of $G(n,p)$ coincides with the threshold for having minimum degree at least $d$, namely $p=(\log{n} + (d- 1)\log\log{n})/n$ (extending work by Jackson, Servatius, and Servatius~\cite{jackson2007rigidity} in the case $d=2$). In~\cite{peled2024rigidity}, Peled and Peleg extended this result by showing that below the threshold $p=(2/(1-\log{2}))\log{n}/n$, $G(n,p)$ is \whp{} $\delta(G)$-rigid, where $\delta(G)$ is the minimum degree of $G$.

    In~\cite{KLM25}, we proved that for every $\eps>0$ there is $c>0$ such that if $p\ge (1+\eps)\log{n}/n$, then $G(n,p)$ is \whp{} $\lfloor cnp/\log{(np)}\rfloor$-rigid. We also conjectured that, for $p=\omega(\log{n}/n)$, $G(n,p)$ is \whp{} $(1-o(1))n(1-\sqrt{1-p})$-rigid. A standard edge-counting argument shows that, if true, this result would be optimal. Peleg and Peled~\cite{peled2024rigidity} resolved this conjecture in the regime $(2/(1-\log{2}))\log{n}/n\le p=o(n^{1/2})$, by showing that for such $p$, $G(n,p)$ is \whp{} $\lfloor(1/2-o(1))np\rfloor$-rigid (note that, for $p=o(1)$, $1-\sqrt{1-p}=(1/2+o(1))p$). In the case when $p$ is a constant, Jord\'an, Liu, and Vill\'anyi~\cite{jordan2025degree} very recently resolved the conjecture up to a constant factor, by showing that for every fixed $p\in(0,1)$ there is $c>0$ such that $G(n,p)$ is \whp{} $\lfloor c n\rfloor$-rigid. Here, we extend this result, solving our conjecture, up to a multiplicative constant, for all relevant values of $p$.

\begin{theorem}[Binomial random graphs]\label{thm:gnp}
   For every $\eps>0$ there exists $c=c(\eps)>0$ such that if $p\ge (1+\eps)\log{n}/n$, then $G(n,p)$ is \whp{} $d$-rigid for $d=\lfloor c n p\rfloor$.
\end{theorem}

Note that for, e.g., $p\ge 2\log{n}/n$,
\Cref{thm:gnp} implies that
$G(n,p)$ is \whp{} $d$-rigid
for $d=\floor{cnp}$, where $c>0$ is a {\em universal} constant. In fact, we may choose $c=1/4000$ in this case (see proof of \cref{thm:gnp} for details).

We say that a set $U\subseteq V$ is a \emph{$d$-rigid component} of $G$ if $G[U]$ is $d$-rigid but, for all $U'\supsetneq U$, $G[U']$ is not $d$-rigid.
A classical result on random graphs states that $p=1/n$ is the threshold for the appearance of a ``giant'' (that is, linearly sized) connected component in a random binomial graph $G(n,p)$ (see, for example,~\cite{bollobas2001book}). It is natural to consider the high-dimensional geometric version of this question; that is, to consider the threshold for the appearance of a giant $d$-rigid component.
Lew, Nevo, Peled and Raz~\cite{LNPR23} conjectured that the phase transition occurs exactly when the average degree of the $(d+1)$-core of the graph exceeds $2d$, which happens at $p\sim d/n$. The case $d=2$ was proved, earlier, by Kasiviswanathan, Moore and Theran in~\cite{kasiviswanathan2011rigidity}. For general $d$, we proved in~\cite{KLM25} that for every $\eps>0$ there is a constant $C>0$ such that, for every $d\ge 2$, the graph $G(n,Cd \log{d}/n)$ contains \whp{} a $d$-rigid component of size at least $(1-\eps)n$. Here, we improve this result, obtaining the following condition, which coincides, up to a multiplicative constant, with the conjectural threshold $p\sim d/n$.

\begin{theorem}[Giant rigid component]\label{thm:rigid_component_gnp:2}
Let $100/n\le p\le 2\log{n}/n$, and let $G\sim G(n,p)$. Then, \whp{}, there exists $U\subseteq [n]$ with $|U|\ge (1-e^{-np/2})n$, such that $G[U]$ is $d$-rigid for $d=\lfloor np/251\rfloor$.
\end{theorem}

Let $n,r\ge 1$ be such that $nr$ is even. Let $G_{n,r}$ be the uniform distribution on $r$-regular graphs on vertex set $[n]$. It is conjectured (see, for example,~\cite{KLM25}) that for every $d\ge 2$ and $r\ge 2d$, a graph $G\sim G_{n,r}$ is \whp{} $d$-rigid. This was proved in the special case $d=2$ by Jackson, Servatius, and Servatius in~\cite[Thm. 4.2]{jackson2007rigidity}. For $d\ge 3$, the question remains open. Since, for $r\ge 3$, a random $r$-regular graph is known to be $r$-connected \whp{}~\cites{bollobas2001book,wormald1981connectivity}, Vill\'anyi's results from~\cite{Vil23+} imply that, for fixed $d$ and $r\ge d(d+1)/2$, $G\sim G_{n,r}$ is \whp{} $d$-rigid. In~\cite{KLM25}, we showed that there exists a constant $C>0$ such that every fixed $d$ and $r\ge C d \log{d}$, $G\sim G_{n,r}$ is \whp{} $d$-rigid. Here, we obtain the following improved result.

\begin{theorem}[Random regular graphs]\label{thm:random_regular}
Let $d\ge 2$ be fixed, and let $r\ge 501d$. Let $G\sim G_{n,r}$. Then, \whp{}, $G$ is $d$-rigid.
\end{theorem}

Note that the condition on the degree $r$ in \cref{thm:random_regular} matches the conjectural bound $r\ge 2d$ up to the multiplicative constant.

~

We now turn from random graph models to the extremal problem of establishing
sufficient minimum codegree conditions for rigidity.
As we mentioned, minimum degree conditions for rigidity
have been studied in~\cites{KLM25,krivelevich2024minimum,jordan2016rigidity}.
In~\cite{KLM25}, we showed that $\delta(G)\ge n/2+d-1$ implies $d$-rigidity for $d=O(\sqrt{n}/\log{n})$
(which is sharp up to a multiplicative constant)
and conjectured that this should hold for all $d\le n/2$.
In~\cite{krivelevich2024minimum}, we extended
the regime of $d$ for which the result holds to $O(n/{\log^2}{n})$, and proved a sharp result,
namely, that $\delta(G)\ge (n+d)/2-1$ implies $d$-rigidity,
in the regime $d=O(\sqrt{n})$. Later, Jord\'an, Liu, and Vill\'anyi~\cite{jordan2025degree} studied this problem,
extended the first result to all values of $d$ (see \cref{thm:JLV}),
and the sharp result to $d\le n/29$.

For a graph $G=(V,E)$ and $v\in V$, let $N_G(v)$ be the set of neighbours of $v$ in $G$.
For distinct $u,v\in V$, write $d_2(u,v)=|N_G(u)\cap N_G(v)|$,
and let $\delta_2(G)$ be the \defn{minimum codegree} of $G$,
defined as the minimum of $d_2(u,v)$, taken over all $u\ne v$ in $V$. Note that, if $\delta(G)\ge (n+x)/2$, then $\delta_2(G)\ge x$. This motivates the problem of proving minimum codegree sufficient conditions for rigidity.
As part of the proof of the minimum degree condition in~\cite{KLM25}, we showed that if $\delta_2(G)\ge k$ for $k\ge 3\log{n}$, then $G$ is $d$-rigid for $d=\lfloor k/(3\log{n})\rfloor$ (\cite[Theorem 5.1]{KLM25}). Here, we obtain the following sharper result.

\begin{theorem}[Minimum codegree]\label{thm:codegree}
  There exists $n_0>0$
  such that for all $n\ge n_0$ and $k\ge 65\log{n}$,
  the following holds.
  Let $G$ be an $n$-vertex graph with $\delta_2(G)\ge k$.
  Then, $G$ is $\floor{k/40}$-rigid.
\end{theorem}

This result is sharp up to a multiplicative constant.
Indeed,
let $k\ge 1$ and $n\ge k+4$ be such that $n+k$ is even,
and consider a graph $G$ composed of two cliques,
each of size $(n+k)/2$, intersecting in $k$ vertices. Note that $G$ is an $n$-vertex graph with minimum codegree $k$. Moreover, it is easy to check that $G$ is $k$-rigid but is not $(k+1)$-rigid.

Our proofs combine combinatorial and probabilistic arguments. Our main technical tool is the following sufficient condition for the existence of a large $d$-rigid component in a graph. For a graph $G=(V,E)$ and two disjoint sets $X,Y\subseteq V$, we denote $E_G(X,Y)=E(X,Y)=\{e\in E:\, |e\cap A|=|e\cap B|=1\}$. Recall that $\delta(G)$ denotes the minimum degree of $G$. For an integer $m$, let $[m]=\{1,2,\ldots,m\}$, and let $\binom{[m]}{2}=\{\{i,j\}:\, 1\le i<j\le m\}$. For a finite set $V$, we say that $V_1,\ldots,V_m$ is a \emph{partition} of $V$ if $V=V_1\cup \cdots \cup V_m$, $V_i\cap V_j=\es$ for all $1\le i<j\le m$, and $V_i\ne \es$ for all $1\le i\le m$.

\begin{theorem}\label{thm:connector_partition}
    Let $G=(V,E)$ be a graph, let $k,m$ be positive integers, and let $1/2<\eta\le 1$. Let $V_1,\ldots, V_m$ be a partition of $V$ satisfying $|V_i|>7k-3$ for all $1\le i\le m$. Let $G_0=([m],E')$ be the graph defined by
    \[
        E'= \left\{ \{i,j\}\in\binom{[m]}{2} :\, \text{for every $X\subseteq V_i$, $Y\subseteq V_j$ such that $|X|=|Y|= k$, $E_G(X,Y)\ne \es$} \right\}.
    \]
    Assume that $\delta(G_0)\ge \eta\cdot m-1$.
Then, there exists $W\subseteq V$ such that $G[W]$ is $d$-rigid for $d=\lfloor (\eta-1/2)m\rfloor$, and  $|V_i\sm W|\le 4k$ for all $1\le i\le m$.
\end{theorem}

The proof of \cref{thm:connector_partition} relies (in the case $\eta<1$) on a minimum degree condition for rigidity recently proved by Jord\'an, Liu, and Vill\'anyi~\cite{jordan2025degree} (see \cref{thm:JLV}), as well as on the following new combinatorial sufficient condition for $d$-rigidity, which we expect to be of independent interest.

Let $G=(V,E)$ be a graph. For $A,B\subseteq V$, let $G[A,B]$ be the graph on vertex set $A\cup B$ with edges $\{e\in E: e\cap A\ne \es, \, e\cap B\ne \es, \, e\subseteq A\cup B\}$. In particular, $G[A,A]=G[A]$, the subgraph of $G$ induced by $A$, and if $A$ and $B$ are disjoint then $G[A,B]$ is a bipartite graph whose edge set consists of all the edges of $G$ with one endpoint in $A$ and the other in $B$.

Let $V_1,\ldots, V_m$ be a partition of $V$. We define the \emph{reduced graph} of $G$ associated to this partition as $G'=([m],E')$, where
\[
    E' = \left\{\{i,j\}\in\binom{[m]}{2} :\, E_G(V_i,V_j)\ne \es\right\}.
\]

\begin{theorem}\label{thm:strong_rigid_partition_1}
   Let $d\ge 1$ and $m\ge d$. Let $G=(V,E)$ be a graph, and let $V_1,\ldots, V_m$ be a partition of $V$.
   Assume that the reduced graph $G'$ is $d$-rigid,
   and that for every $1\le i\le m$ and every pair of distinct vertices $u,v\in V_i$, there are at least $d$ distinct indices $j\in[m]$ such that $u$ and $v$ lie in the same connected component of $G[V_i,V_j]$. Then, $G$ is $d$-rigid.
\end{theorem}

\Cref{thm:strong_rigid_partition_1} extends results on the rigidity of graphs admitting ``type-I strong $d$-rigid partitions'' introduced by Lew, Nevo, Peled, and Raz in~\cite{LNPR23+} (corresponding to the special case $m=d$), and ``type-II strong $d$-rigid partitions'', introduced in~\cite{KLM25} (corresponding to the special case when $m=d+1$ and all the edges inside each part $V_i$ are ignored). Moreover, \cref{thm:strong_rigid_partition_1} extends Whiteley's well-known vertex-splitting lemmas~\cites{whiteley90vertex,whiteley96some} (see \cref{lemma:whiteley_splitting_1,lemma:whiteley_splitting_2}),  corresponding to the special case $|V_1|=2$, $|V_i|=1$ for $2\le i\le |V|-1$. Let us note that shortly before the submission of this manuscript, we learned that in recent, yet unpublished work, Csaba Kir\'aly independently obtained, using arguments similar to ours, new proofs of the ``strong rigid partitions'' conditions from~\cites{LNPR23+,KLM25}, as well as to an extended version of them sharing some similarity to \cref{thm:strong_rigid_partition_1}.

In Section \ref{sec:partitions}, we present more general, albeit more technical, versions of \cref{thm:strong_rigid_partition_1} (\cref{thm:gen_partitions,thm:strong_rigid_partition_2}), extending the notion of ``rigid partitions'' introduced in~\cite{KLM25} (originally studied, using a  different terminology, by Crapo~\cite{crapo1990plane} in the $d=2$ case, and by Lindemann~\cite{LinPhD} in the $d=3$ case).

\begin{remark} As we already mentioned, the minimum degree sufficient condition for rigidity obtained by Jord\'an, Liu, and Vill\'anyi in~\cite{jordan2025degree} (see \cref{thm:JLV}) is an important ingredient in our proof of \cref{thm:connector_partition} and, as a consequence, of \cref{thm:gnp}. Indeed, the idea of combining \cref{thm:JLV} with Whiteley's vertex-splitting lemmas, introduced in~\cite{jordan2025degree} in order to study the rigidity of $G(n,p)$ for constant $p$, was a significant motivation for this work.
\end{remark}

The paper is organized as follows.
In \cref{sec:prelims},
we introduce some relevant facts from rigidity theory and probability theory
that we use later.
In \cref{sec:connector_partition},
we prove \cref{thm:connector_partition}.
We establish \cref{thm:gnp} in \cref{sec:binomial},
and prove \cref{thm:rigid_component_gnp:2,thm:random_regular} in \cref{sec:sparse}.
In \cref{sec:codegree} we prove \cref{thm:codegree},
providing a minimum codegree condition for rigidity.
We defer the proof of \cref{thm:strong_rigid_partition_1}
and its generalizations (\cref{thm:gen_partitions,thm:strong_rigid_partition_2})
to \cref{sec:partitions}.
We conclude in \cref{sec:concluding} with some final remarks and open questions.

\paragraph{Notation and terminology}
Let $G=(V,E)$ be a graph.
Denote by $N_G(A)$ the
{\em external} neighbourhood of $A$,
that is, the set of all vertices in $V\sm A$
that have a neighbour in $A$.
The degree of a vertex $v\in V$,
denoted by $\deg_G(v)$,
is its number of incident edges.
In the above notation we often replace $\{v\}$ with $v$ for abbreviation,
and often omit the subscript $G$ if it is clear from the context.
For a set $X$ and a nonnegative integer $k$,
we denote by $\binom{X}{k}$ the family of all $k$-subsets of $X$.
Throughout the paper, all logarithms are in the natural basis.

\section{Preliminaries}\label{sec:prelims}

\subsection{Inductive operations preserving rigidity}

We will need the following basic operation preserving $d$-rigidity of graphs.

\begin{lemma}[$0$-extension; see, for example,~\cite{TW1985}]
\label{lemma:0_extension}
Let $G=(V,E)$ be a $d$-rigid graph, and let $G'$ be obtained from $G$ by adding a new vertex $v$ adjacent to at least $d$ vertices from $V$. Then, $G'$ is $d$-rigid.
\end{lemma}

In addition, let us recall the precise statement of Whiteley's vertex-splitting lemmas.

\begin{lemma}[Whiteley~\cite{whiteley90vertex}]\label{lemma:whiteley_splitting_1}
  Let $G=(V,E)$ be a $d$-rigid graph, and let $v\in V$.
  Let $G'=(V',E')$ be obtained from $G$ by removing $v$ and adding two new vertices $x$ and $y$ such that $N_{G'}(x)\cup N_{G'}(y)= N_G(v)\cup\{x,y\}$, and $|N_{G'}(x)\cap N_{G'}(y)|\ge d-1$. Then, $G'$ is $d$-rigid.
\end{lemma}

\begin{lemma}[Whiteley~\cite{whiteley96some}]\label{lemma:whiteley_splitting_2}
Let $G=(V,E)$ be a $d$-rigid graph, and let $v\in V$.
  Let $G'=(V',E')$ be obtained from $G$ by removing $v$ and adding two new vertices $x$ and $y$ such that $N_{G'}(x)\cup N_{G'}(y)= N_G(v)$, and $|N_{G'}(x)\cap N_{G'}(y)|\ge d$. Then, $G'$ is $d$-rigid.
\end{lemma}

\subsection{Minimum degree condition for rigidity}

We will also need the following minimum degree sufficient condition for $d$-rigidity, proved under the assumption $d=O(n/{\log}^2{n})$ in~\cite{krivelevich2024minimum}, and in the general case in~\cite{jordan2025degree}. Recall that, for a graph $G=(V,E)$, we denote by $\delta(G)$ the minimum degree of a vertex in $G$.

\begin{theorem}[Jord\'an, Liu, Vill\'anyi {\cite[Theorem 1.2]{jordan2025degree}}]
\label{thm:JLV}
  Let $1\le d<n$ be integers.
  If $G$ is an $n$-vertex graph with $\delta(G)\ge \frac{n}{2}+d-1$,
  then $G$ is $d$-rigid.
\end{theorem}

\subsection{Concentration inequalities}
We will repeatedly make use of the following version of Chernoff bounds
(see, e.g., in~\cite{JLR}*{Chapter~2}).
Let $\phi(x)=(1+x)\log(1+x)-x$ for $x>-1$.
\begin{theorem}[Chernoff bounds]\label{thm:chernoff}
  Let $n\ge 1$ be an integer and let $p\in[0,1]$,
  let $X\sim\Bin(n,p)$, and let $\mu=\E{X}=np$.
  Then, for every $\nu>0$,
  \begin{align*}
    \pr(X\le \mu-\nu) &\le \exp\left(-\mu\phi\left(-\frac{\nu}{\mu}\right)\right) \le \exp\left(-\frac{\nu^2}{2\mu}\right),\\
    \pr(X\ge \mu+\nu) &\le \exp\left(-\mu\phi\left(\frac{\nu}{\mu}\right)\right) \le \exp\left(-\frac{\nu^2}{2(\mu+\nu/3)}\right).
  \end{align*}
\end{theorem}

\section{A combinatorial sufficient condition for the existence of a large $d$-rigid component}\label{sec:connector_partition}

Here, we prove \cref{thm:connector_partition}, which will be a key ingredient in the proofs of \cref{thm:gnp,thm:rigid_component_gnp:2,thm:random_regular}.

\begin{proof}[Proof of \cref{thm:connector_partition}]
Let $s=4k$. For $1\le i<j\le m$, we call a set $C\subseteq V_i\cup V_j$ {\em big} if $|C\cap V_i|\ge k$ and $|C\cap V_j|\ge k$.

\begin{claim}\label{claim:unique_big_component}
    Let $\{i,j\}\in E'$, and let $U_i\subseteq V_i$ and $U_j\subseteq V_j$ such that $|U_i|,|U_j|\le s$. Then, the subgraph $G[V_i\sm U_i,V_j\sm U_j]$ contains a unique big connected component.
\end{claim}
\begin{proof}
   It is immediate that we cannot have two big components $C,C'$, as otherwise, since $\{i,j\}\in E'$, $G$ has an edge between $C\cap V_i$ and $C'\cap V_j$, a contradiction to $C$ and $C'$ being distinct components.

 To prove there is at least one such component, notice that if there is a component $C$ with $|C\cap V_i|\ge k$, then,  since $\{i,j\}\in E'$, the set $C\cap V_i$ is connected to at least $|V_j\sm U_j|-k+1\ge k$ vertices in $V_j\sm U_j$, providing the desired component; similarly for a component with at least $k$ vertices in $V_j$. Assume thus for contradiction that all components in $G[V_i\sm U_i, V_j\sm U_j]$ have less than $k$ vertices on each side.
Let ${\cal C}_0$ be a minimal by inclusion collection of connected components in this graph whose union contains at least $k$ vertices in at least one of the parts $V_i$ or $V_j$. Without loss of generality, assume that the union of the components in ${\cal C}_0$ has at least $k$ vertices in $V_i$. Then, by the minimality of $\mathcal{C}_0$, it has at most $2k-2$ vertices in $V_j$, and thus, since $\{i,j\}\in E'$, the vertices of ${\cal C}_0$ in $V_i$ are connected to at least $|V_j\sm U_j|-(2k-2)-k+1 = |V_j|-s-3k+3>0$ vertices in $V_j\sm U_j$ outside of ${\cal C}_0$, a contradiction to $\mathcal{C}_0$ being a union of connected components.
\end{proof}

For $\{i,j\}\in E'$ and $U\subseteq V$ such that $|U\cap V_i|, |U\cap V_j|\le s$, denote by $C_{ij}(U)$ the unique big connected component in $G[V_i\sm U,V_j\sm U]$, which exists by Claim \ref{claim:unique_big_component}. Notice that, by the uniqueness of the big components, we have, for all $U\subseteq U'$, $C_{ij}(U')\subseteq C_{ij}(U)$.

For $1\le i\le m$, let $T_i=\deg_{G_0}(i)(1-k/s)$.
Let $U\subseteq V$ such that $|U\cap V_i|\le s$ for all $1\le i\le m$.
We say that a vertex $v\in V_i\sm U$, $1\le i\le m$, is \emph{bad} for $U$ if $v$ is contained in less than $T_i$ of the big components $C_{ij}(U)$, for $j\in N_{G_0}(i)$. Note that, if $v$ is bad for $U$, then it is bad for all $U'\supset U$ with $v\notin U'$ (indeed, assuming $v\in V_i$, for every $j$ such that $v\notin C_{ij}(U)$, we must have $v\notin C_{ij}(U')\subseteq C_{ij}(U)$).

\begin{claim}\label{claim:good_set}
    There exists $U_0\subseteq V$ satisfying $|U_0\cap V_i|\le s$ for all $1\le i\le m$, such that there are no bad vertices for $U_0$.
\end{claim}
\begin{proof} Let
\[
    \mathcal{U} = \left\{U\subseteq V:\, |U\cap V_i|\le s \text{ for all } i\in[m], \text{ and } \text{$v$ is bad for $U\sm \{v\}$ for all $v\in U$}\right\}.
\]
Note that $\mathcal{U}$ is non-empty, as $\es\in \mathcal{U}$.
Let $U_0\subseteq V$ be a maximal set in $\mathcal{U}$. We will show that there are no bad vertices for $U_0$.  Assume for contradiction that there exist $1\le i\le m$ and $v\in V_i\sm U$ which is bad for $U_0$. If $|U_0\cap V_i|<s$, then $U'=U_0\cup\{v\}$ belongs to $\mathcal{U}$. Indeed, $v$ is bad for $U'\sm \{v\}=U_0$ by definition, and for $u\in U_0$, $u$ is bad for $U_0\sm\{u\}$, and therefore it is also bad for $U'\sm\{u\}\supset U_0\sm\{u\}$. But this contradicts the maximality of $U_0$.

Therefore, we must have $|U_0\cap V_i|=s$. By the definition of $\mathcal{U}$, every $u\in U_0\cap V_i$ does not belong to at least $\deg_{G_0}(i)-T_i$ of the components $C_{ij}(U_0\sm\{u\})$, $j\in N_{G_0}(i)$. Note that, in such case, $C_{ij}(U_0\sm \{u\})=C_{ij}(U_0)$, and $u$ is not adjacent to any vertex in $C_{ij}(U_0)$. By double counting, there must be some $j\in N_{G_0}(i)$ and a set $A\subseteq V_i\cap U_0$ with
\[
|A|\ge s(\deg_{G_0}(i)-T_i)/\deg_{G_0}(i)= s(1-T_i/\deg_{G_0}(i))= k,
\]
such that none of the vertices of $C_{ij}(U_0)$ is adjacent to any of the vertices in $A$.
But, by definition, $|C_{ij}(U_0)\cap V_j|\ge k$, a contradiction to $\{i,j\}\in E'$.
\end{proof}

Now, let $U_0\subseteq V$ be a set satisfying $|U_0\cap V_i|\le s$ for all $1\le i\le m$ and such that there are no bad vertices for $U_0$, whose existence is guaranteed by Claim \ref{claim:good_set}. Then, for every $1\le i\le m$ and  $u,v\in V_i\sm U_0$, $u$ and $v$ belong each to at least $T_i$ big components $C_{ij}(U_0)$, for $j\in N_{G_0}(i)$, and therefore they belong to at least $2T_i-\deg_{G_0}(i)=
\deg_{G_0}(i)(1-2k/s)\ge \delta(G_0)(1-2k/s)=\delta(G_0)/2$  common big components (where we used $s=4k$). Note that, for $\eta\le 1$,
\[
(\eta-1/2)m - (\eta m-1)/2 = (\eta-1)m/2+1/2 <1,
\]
and so
\[
\lceil \delta(G_0)/2\rceil \ge \lceil (\eta m-1)/2 \rceil \ge \lfloor (\eta-1/2)m\rfloor=d.
\]
That is, each $u,v\in V_i$ are connected by a path in $G[V_i\sm U_0,V_j\sm U_0]$ for at least $d$ different indices $j\in N_{G_0}(i)$.

In addition, since $|V_i\sm U_0|> 3k-3\ge k$ for all $1\le i\le m$, there is an edge between $V_i\sm U_0$ and $V_j\sm U_0$ for all $\{i,j\}\in E'$. In other words, the reduced graph $G'$ associated with the partition $\{V_i\sm U_0\}_{i=1}^m$ of $G[V\sm U_0]$ contains $G_0$ as a (spanning) subgraph. By \cref{thm:JLV}, $G_0$ is $d$-rigid, and therefore $G'$ is $d$-rigid as well. Hence, by \cref{thm:strong_rigid_partition_1}, $G[V\sm U_0]$ is $d$-rigid, as wanted.
\end{proof}

\begin{remark}
    Note that if $\eta=1$, then the graph $G_0$ is a complete graph, and therefore it is $(m-1)$-rigid (and in particular, $d$-rigid for $d=\lfloor m/2\rfloor$). Thus, the special case $\eta=1$ of \cref{thm:connector_partition} (which we use in the proofs of \cref{thm:rigid_component_gnp:2,thm:random_regular}) can be proved without applying \cref{thm:JLV}.
\end{remark}

Given a graph $G=(V,E)$ and a vertex set $U\subseteq V$,
denote $\partial{U}=E(U,V\sm U)$. For $k\ge 1$, let
\[
  \iso(G;k) = \min_{\substack{U\subseteq V,\\k\le |U|\le |V|/2}} \frac{|\partial U|}{|U|}.
\]
Let $\iso(G)=\iso(G;1)$ be the \defn{isoperimetric number} of $G$.

The following lemma will be useful.
\begin{lemma}[Absorption]
\label{lemma:absorption}
  Let $d\ge 1$, $k\ge 1$,
  and let $G=(V,E)$ be a graph with $\iso(G;k)\ge d$.
  Let $B\subseteq V$ such that
  $|B|\ge|V|/2$ and $G[B]$ is $d$-rigid.
  Then, there is $B'\supseteq B$ with $|B'|> |V|-k$ such that $G[B']$ is $d$-rigid.
\end{lemma}

\begin{proof}

  Let $W\supseteq B$ be a maximum-sized set for which $G[W]$ is $d$-rigid. Assume for contradiction that $|W|\le |V|-k$.
  Write $U=V\sm W$.
  By assumption, we know that $k\le |U|\le |V|/2$.
  By definition, $|\partial{U}|\ge \iso(G;k)\cdot|U|\ge d|U|$,
  hence, by averaging, there exists $u\in U$ with at least $d$ neighbours in $W$.
  Finally, by \cref{lemma:0_extension}, $G[W\cup\{u\}]$ is $d$-rigid, contradicting the maximality of $W$.
\end{proof}

\section{Rigidity of binomial random graphs}\label{sec:binomial}

We will need the following lemma about edge expansion in binomial random graphs. Our argument is fairly standard (see, for example,~\cite{benjamini2008isoperimetric} and the references within).

\begin{lemma}\label{lemma:gnp_expansion_general}
  For every $\eps\in(0,1]$
  there exist constants $C=C(\eps)>0$ and $\kappa=\kappa(\eps)>0$
  such that the following holds.
  If $p\ge C/n$ and $G\sim G(n,p)$, then, \whp{}, $\iso(G;k_0)\ge \kappa n p$ for $k_0=\ceil{n\exp(-(1-\eps/2)np)}$. For $\eps=1$, we may take $\kappa=1/30$ and $C=40$.
\end{lemma}

\begin{proof}
  Let $\kappa=\kappa(\eps)\in(0,1/2)$ be a constant to be determined later. Let $C=40/\eps$. For $U\subseteq [n]$, let $X_U$ be the event that $|E(U,[n]\sm U)|< \kappa  n p |U|$.
  For $1\le k\le n$, denote $p_k=\pr(X_U)$ for $U\in \binom{[n]}{k}$ (note that $\pr(X_U)$ depends only on $k$).
  Let
  \[
    X=\bigcup_{U\subseteq[n]:\ k_0\le |U|\le n/2} X_U.
  \]
  We show that $\pr(X)=o(1)$.

  Let $k_0\le k\le n/2$ and let $U\in\binom{[n]}{k}$.
  Let $Y_U=|E(U,[n]\sm U)|$.
  Note that $Y_U$ is a binomial random variable with $k(n-k)$ attempts and success probability $p$, and so
  \[
    \mu=\E(Y_U)=k(n-k)p.
  \]
  Let $\nu=\mu-\kappa n p k$.
  Since $k\le n/2$ we have $n-k\ge n/2$, and
  \[
    \frac{\nu}{\mu}=1-\frac{\kappa n}{n-k}\ge 1-2\kappa.
  \]
  By Chernoff bounds (\cref{thm:chernoff}) and the fact that $\phi(x)=(1+x)\log(1+x)-x$ is decreasing in $(-1,0)$,
  \begin{align*}
    p_k
    &=\pr(X_U)=\pr(Y_U<\kappa npk)=\pr(Y_U<\mu-\nu) \\
    &\le \exp(-\mu\phi(-\nu/\mu))
    \le \exp(-\mu\phi(-(1-2\kappa))).
  \end{align*}
  Since $\phi(x)$ converges to $1$ as $x\searrow -1$,
  we may choose $\kappa$ so that
  \[
    \phi(-(1-2\kappa))\ge 1-\eps/4.
  \]
(In the special case $\eps=1$, $\kappa=1/30$ satisfies this condition.) Hence, for all $k_0\le k\le n/2$,
  \[
    p_k\le \exp(-(1-\eps/4)pk(n-k)).
  \]

  We now split into two cases.
  \begin{description}
    \item[Case 1: $k_0\le k\le (\eps/8)n$.]
      In this case,
      \[
        \begin{aligned}
          \binom{n}{k}p_k
          &\le \left(\frac{en}{k}\right)^k \exp\left(-(1-\eps/4)pk(1-\eps/8)n\right)\\
          &\le \left(\frac{en}{k_0}\right)^k \exp\left(-(1-\eps/4)pk(1-\eps/8)n\right)\\
          &\le \exp(k(1+(1-\eps/2)np - (1-\eps/4)(1-\eps/8)np))\\
          &\le \exp(k(1-(\eps/8)np))
          \le \exp(-(\eps/16)npk).
        \end{aligned}
      \]
      In the last step, we used the fact that $np>C=40/\eps>16/\eps$.
    \item[Case 2: $(\eps/8)n < k\le n/2$.]
      In this case, $k(n-k)\ge (\eps/8)n\cdot (n/2)=(\eps/16)n^2$, so
      \[
        p_k\le \exp(-(1-\eps/4)(\eps/16) n^2 p)
            = \exp(-\gamma  n^2 p),
      \]
      where $\gamma=(1-\eps/4)(\eps/16)>0$.
      Using $\binom{n}{k}\le 2^n$, we get
      \[
        \binom{n}{k}p_k \le \exp(n\log 2-\gamma  n^2 p).
      \]
      Since $C= 40/\eps$, we have $\gamma n p \ge \gamma C= 40\gamma/\eps>\log 2+1$,
      and hence $\binom{n}{k}p_k\le e^{-n}$ for all $k$.
  \end{description}

  Before proceeding, note that $npk_0=\omega(1)$.
  Indeed, if $np\ge\log{n}$ then $npk_0\ge\log{n}=\omega(1)$;
  and, noting that $k_0$ is decreasing in $np$,
  if $np\le\log{n}$ then $k_0\ge n^{\eps/2}=\omega(1)$,
  which implies (since $np\ge C$) that $npk_0=\omega(1)$.
  Write $\rho=\frac{\eps}{16}np\ge 1$,
  so $e^{-\rho k_0}=o(1)$.
  Combining the two cases above and applying the union bound, we obtain
  \begin{align*}
    \pr(X)
    &\le \sum_{k=k_0}^{\lfloor (\eps/8)n\rfloor} \binom{n}{k}p_k
      +\sum_{k=\lfloor (\eps/8)n\rfloor+1}^{\lfloor n/2\rfloor} \binom{n}{k}p_k \\
    &\le \sum_{k=k_0}^{\infty}\exp(-\rho k) + n e^{-n}
    = \frac{e^{-\rho k_0}}{1-e^{-\rho}} + o(1)
     = o(1),
  \end{align*}
  as required.
\end{proof}

\begin{corollary}\label{lemma:gnp_expansion}
  For every $\eps>0$ there is $\kappa=\kappa(\eps)\in(0,1)$
  such that for $p\ge (1+\eps)\log{n}/n$,
  a graph $G\sim G(n,p)$ satisfies $\iso(G)\ge \kappa n p$ \whp{}.
\end{corollary}

\begin{proof}
  By monotonicity in $p$, we may assume $\eps\in(0,1/2)$.
  Apply \cref{lemma:gnp_expansion_general} with parameter $\eps$
  to obtain $\kappa=\kappa(\eps)\in(0,1)$.
  Since $np\ge (1+\eps)\log n$, we have $k_0=1$ for sufficiently large $n$.
  Thus, \whp{}, $\iso(G)=\iso(G;1)\ge \kappa n p$.
\end{proof}

An \defn{equipartition} $X_1,\dots,X_m$ of a set $X$
is a partition of $X$ into $m$ parts of size $\lceil |X|/m \rceil$ or $\lfloor |X|/m\rfloor$ each. For convenience, we assume $|X_1|\ge \cdots \ge |X_m|$.

\begin{proof}[Proof of \cref{thm:gnp}]
    Let $G\sim G(n,f/n)$, where $f=np$. For every integer $k$,  the probability that there are no edges in $G$ between two fixed disjoint sets $X$ and $Y$ of size $k$ each is $(1-f/n)^{k^2}\le e^{-f k^2/n}$. Let $k=\lceil \beta n/f\rceil$ where $\beta>0$ is some large enough constant to be chosen later, and let $m=\lceil \alpha f\rceil$, where $\alpha= 1 /(56 \beta)$. Let $V_1,\ldots, V_m$ be an arbitrary fixed equipartition of $[n]$.

    First, note that, for every $1\le i\le m$,
    \[
        |V_i|\ge \lfloor n/m\rfloor \ge n/(2m) \ge n/(4\alpha f) = 14 \beta n/f \ge 7 k.
    \]

    Let $G_0=([m],E')$, where
  \[
        E'= \left\{ \{i,j\}\in\binom{[m]}{2} :\, \text{for every $X\subseteq V_i$, $Y\subseteq V_j$ such that $|X|=|Y|= k$, $E_G(X,Y)\ne \es$} \right\}.
    \]
    For $1\le i<j\le m$, let $X_{ij}$ be the event that $\{i,j\}\in E'$. Note that the events $X_{ij}$ are mutually independent, as the events depend on pairwise disjoint sets of edges of $G$. Moreover, by the union bound,
    \begin{align*}
      1-\pr(X_{ij}) &\le \binom{\lceil n/m\rceil}{k}^2  e^{-f k^2/n}
        \le \left(\frac{ e }{k }\left\lceil\frac{n}{m}\right\rceil
        \right)^{2k} \cdot e^{-f k^2/n}
        \le \left(\frac{ 2e n }{m k }
        \right)^{2k} \cdot e^{-f k^2/n}
        \\ &
        =\left(\frac{ 2e n }{m k }
         \cdot e^{-f k/(2n)}\right)^{2k}
        \le
        \left( \frac{2e}{\alpha \beta} e^{-\beta/2} \right)^{2k} \le (305 e^{-\beta/2})^{2k}.
    \end{align*}
    Choosing $\beta= 17$, we obtain $305 e^{-\beta/2}\le 1/10$, and therefore
    \[
        1-\pr(X_{ij})\le (1/10)^{2k}\le 1/100.
    \]
    That is, for all $1\le i<j\le m$,
    \[
        \pr(X_{ij})\ge 99/100.
    \]
    Hence, for every $i\in[m]$, the degree of $i$ in $G_0$ stochastically dominates a binomial random variable with $m-1$ attempts and success probability $99/100$.
    By Chernoff bounds (\cref{thm:chernoff}),
  \[
  \pr(\deg_{G_0}(i)<98 m/100)=e^{-\Theta(m)}.
  \]
  By the union bound over all $i\in[m]$, we obtain that $\delta(G_0)\ge 98m/100$ \whp{}. So, by \cref{thm:connector_partition} (with $\eta=0.98)$, \whp{} there exists $U\subseteq [n]$ with
  \[
  |U|\ge n-m\cdot 4k\ge n- 16 \alpha \beta n > 0.7 n,
  \]
  such that $G[U]$ is $d'$-rigid for
  \[
  d'= \lfloor (98/100-1/2) m\rfloor \ge 0.24 \alpha f \ge f/4000.
  \]

  By \cref{lemma:gnp_expansion}, there exists $\kappa=\kappa(\eps)\in(0,1)$ such that $\iso(G)\ge \kappa f$ \whp{}.
Let us choose $c=\min\{1/4000,\kappa\}$, and let $d=\lfloor c f\rfloor$. Note that $\iso(G)\ge d$, and, since $d\le d'$, $G[U]$ is $d$-rigid. Hence, by \cref{lemma:absorption} (applied for $k=1$), $G$ is $d$-rigid, as wanted.
\end{proof}

\section{Giant rigid components in $G(n,p)$ and rigidity of random regular graphs}\label{sec:sparse}

In this section, we study the rigidity of sparse random graphs, proving \cref{thm:rigid_component_gnp:2,thm:random_regular}.  We will need the next result, which follows as an application of \cref{thm:connector_partition}.

\begin{proposition}\label{thm:rigid_component}
   Let $\mathcal{G}$ be a distribution of graphs on vertex set $[n]$, and let $G\sim \mathcal{G}$.  Let $0<f=o(n/\log{n})$, and  assume that for every $K\ge 1$ and every fixed pair of disjoint sets $X$ and $Y$ of size $|X|=|Y|=K$,
   \[
   \pr(E_G(X,Y)=\es)\le \exp(-f K^2/n).
   \]
   Let $0<\eps<4/7$, and assume that $f\ge 16 \ln{(25/\eps})/\eps$.
   Then, \whp{}, there exists $U\subseteq[n]$, $|U|\ge (1-\eps)n$, such that $G[U]$ is $d$-rigid for $d=\lfloor \eps f / (32 \log{(25/\eps)})\rfloor$.
\end{proposition}
\begin{proof}
Let $k=\lceil\beta n/f\rceil$, where $\beta=2 \ln{(25/\eps)}$.
    Let $m= \lceil\alpha f\rceil$, where $\alpha= \eps/(8\beta)$.
    Let $V_1,\ldots,V_m$ be an arbitrary fixed equipartition of $[n]$ (that is, each part $V_i$ has size $\lfloor n/m \rfloor$ or $\lceil n/m\rceil$).

    Let $G\sim \mathcal{G}$. Let $G_0=([m],E')$, where
    \[
        E'= \left\{ \{i,j\}\in\binom{[m]}{2} :\, \text{for every $X\subseteq V_i$, $Y\subseteq V_j$ such that $|X|=|Y|= k$, $E_G(X,Y)\ne \es$} \right\}.
    \]
    By assumption, the probability that there are no edges in $G$ between two fixed disjoint sets $X$ and $Y$ of size $k$ each is at most $\exp(-f k^2/n)$.
    By the union bound over all $1\le i<j\le m$ and all choices of subsets $X\subseteq V_i$, $Y\subseteq V_j$ with $|X|=|Y|=k$, the probability $p'$ that $G_0$ is not the complete graph satisfies
    \begin{align*}
        &p'\le m^2 \binom{\lceil n/m\rceil}{k}^2  e^{-f k^2/n}
        \le m^2 \left(\frac{ e }{k }\left\lceil\frac{n}{m}\right\rceil
        \right)^{2k} \cdot e^{-f k^2/n}
        \le n^2 \left(\frac{ 3 n }{m k }
        \right)^{2k} \cdot e^{-f k^2/n}
        \\ &
        =n^2 \left(\frac{ 3 n }{m k }
         \cdot e^{-f k/(2n)}\right)^{2k}
        \le
        n^2 \left( \frac{3}{\alpha \beta} e^{-\beta/2} \right)^{2k}.
    \end{align*}
    Since $\alpha=\eps/(8\beta)$ and $\beta> 2 \ln(24/\eps)$, we have
    \begin{equation}\label{eq:ineq1}
    3e^{-\beta/2}/(\alpha \beta)=  24 e^{-\beta/2}/\eps<1,
    \end{equation}
    and therefore (using the fact that $k=\Theta( n/f)=\omega(\log{n})$), we obtain that $p'\to 0$ as $n\to \infty$. That is, \whp{}, $G$ is complete (and in particular, it has minimum degree equal to $\eta m-1$, for $\eta=1$). In order to apply \cref{thm:connector_partition}, we are left to show that $\lfloor n/ m \rfloor > 7k-3$. Indeed, since $\alpha\beta=\eps/8$ and $\eps<4/7$, and using the fact that $\alpha f\ge 1$ (as $f \ge 16\ln{(25/\eps}) = 1/\alpha$), we obtain, for large enough $n$,
    \[
\lfloor n/m\rfloor \ge n/m-1 \ge n/(2\alpha f) -1 = 4n\beta/(\eps f)-1 > 7k-1,
    \]
    as required.

    Thus, by \cref{thm:connector_partition}, there is a set $U\subseteq[n]$ of size $|U|\ge n-m\cdot 4k$ such that $G[U]$ is $d'$-rigid for $d'=\lfloor m/2\rfloor  \ge \lfloor \eps f/ (32 \log{(25/\eps)})\rfloor=d$.
    Since $m< 2\alpha f$ as $\alpha f\ge 1$, we obtain, for large enough $n$, $4mk< 4\cdot (2\alpha f)(\beta n/f)=8\alpha \beta n$, and therefore
    \[
    |U|\ge n-m\cdot 4k \ge n-8\alpha \beta n = (1-\eps)n,
    \]
    as wanted.
\end{proof}

\begin{proof}[Proof of \cref{thm:rigid_component_gnp:2}]

Let $100\le f \le 2\log{n}$, and let $G\sim G(n,f/n)$.
Note that, for every $k\ge 1$, the probability that there are no edges in $G$ between two fixed disjoint sets $X$ and $Y$, each of size $k$, is $(1-f/n)^{k^2}\le e^{-f k^2/n}$. Let $\eps=1/2$,  $c=1/251 \le \eps/(32\log{(25/\eps}))$, and $d=\lfloor c f\rfloor$. Note that $f\ge 16 \ln(25/\eps)/\eps$. Hence, by \cref{thm:rigid_component}, there exists \whp{} $U\subseteq[n]$ with $|U|\ge n/2$, such that $G[U]$ is $d$-rigid. By \cref{lemma:gnp_expansion_general}, $G$ satisfies, \whp{},  $\iso(G;k_0)\ge f/30\ge d$, for $k_0=\lfloor n e^{-f/2}\rfloor$.
Therefore, by \cref{lemma:absorption}, there exists $W\subseteq V$ with $|W|\ge n-k_0\ge (1-e^{-f/2})n$, such that $G[W]$ is $d$-rigid.
\end{proof}

In order to prove \cref{thm:random_regular}, we will need the next auxiliary results.
The following lemma is implicit in the proof of~\cite[Lemma 4.9]{KLM25}.
\begin{lemma}\label{lemma:gnr_connector}
Let $G\sim G_{n,r}$, and let $A,B\subseteq[n]$ with $A\cap B=\es$, $|A|=|B|=S$. Then,
\[
\pr(E_G(A,B)=\es) \le (1-S/n)^{r S/2}\le \exp(-r S^2/(2n)).
\]
\end{lemma}

We will also need the following result of Bollob\'as.

\begin{lemma}[Bollob\'as~\cite{bollobas1988isoperimetric}]\label{lemma:bollobas}
    Fix $r\ge 3$, and let $\eta\in(0,1)$ such that
    \[
        2^{4/r}<(1-\eta)^{1-\eta} (1+\eta)^{1+\eta}.
    \]
    Let $G\sim G_{n,r}$. Then, \whp{}, $\iso(G)\ge (1-\eta)r/2$.
\end{lemma}
Note that, for $r\ge 45$, we may take $\eta=1/4$ in \cref{lemma:bollobas}.

\begin{proof}[Proof of \cref{thm:random_regular}]
    Let $d\ge 2$ be fixed, let $r\ge 501 d$, and let $G\sim G_{n,r}$.
    By \cref{lemma:gnr_connector}, we may apply \cref{thm:rigid_component} with parameters $f=r/2$ and $\eps=1/2$ (noting that $f \ge 16 \log(25/\eps)/\eps= 32\log{50}$ in this case). So, by \cref{thm:rigid_component}, \whp{} there exists $U\subseteq [n]$, $|U|\ge n/2$, such that $G[U]$ is $d'$-rigid for $d'=\lfloor \eps f/(32\log{(25/\eps)})\rfloor
    =\lfloor r/(128 \log{50})\rfloor\ge d$. That is, $G[U]$ is $d$-rigid. By \cref{lemma:bollobas}, \whp{} $\iso(G)\ge 3r/8\ge d$. Hence, by
    \cref{lemma:absorption} (applied for $k=1$), $G$ is $d$-rigid.
\end{proof}

\section{Minimum codegree conditions for rigidity}\label{sec:codegree}

In this section, we prove \cref{thm:codegree}.
For a graph $G=(V,E)$,
a partition $\pi=(V_1,\dots,V_m)$ of $V$,
and a pair of vertices $u,v\in V$,
denote $M_\pi(u,v)=\{i\in[m]:N_G(u)\cap N_G(v)\cap V_i\ne\es\}$.
That is, $M_\pi(u,v)$ is the set of indices $i$ for which $u,v$ have a common neighbour in $V_i$.
Write $\delta_2(\pi)=\min_{u,v\in V}|M_\pi(u,v)|$.
Define the \defn{reduced graph} $G_\pi=([m],E_\pi)$,
where $\{i,j\}\in E_\pi$ if and only if $E_G(V_i,V_j)\ne\es$.
Note for later that for each $\pi$,
$\delta(G_\pi)\ge\delta_2(G_\pi)\ge\delta_2(\pi)-2$.

\begin{lemma}\label{lem:codegree}
  For every $\eps\in(0,1)$, $c\in(0,1/(4\log(3/\eps)))$, and $C>8/\eps$,
  the following holds.
  Let $n\ge 1$, $k\ge C\log{n}$, and $m=\floor{ck}$.
  Let $G=(V,E)$ be an $n$-vertex graph with $\delta_2(G)\ge k$,
  and let $\pi$ be a uniform random equipartition of $V$ into $m$ parts.
  Then, \whp{}, $\delta_2(\pi)\ge(1-\eps)m$.
\end{lemma}

\begin{proof}
  Write $n_i=\sum_{j=1}^i|V_j|$ for $i=0,\dots,m$
  (recalling that $|V_1|\ge\dots\ge|V_m|$ are deterministic).
  We generate $\pi$
  by sampling a uniform random permutation of $V$
  and putting each vertex whose location lies in $[n_{i-1}+1,n_i]$ into $V_i$.
  For every $S\subseteq[m]$ denote $V_S=\bigcup_{i\in S}V_i$.
  For every $u,v\in V$
  let $\cM_S(u,v)$ be the event that $M_\pi(u,v)\cap S=\es$.
  Recall that we denote by $d_2(u,v)$ the number of common neighbours of $u$ and $v$.
  Then,
  \[\begin{aligned}
    \pr(\cM_S(u,v))
    &= \frac{\binom{n-|V_S|}{d_2(u,v)}}{\binom{n}{d_2(u,v)}}
    = \prod_{t=0}^{d_2(u,v)-1} \frac{n-|V_S|-t}{n-t}\\
    &\le \prod_{t=0}^{d_2(u,v)-1} \frac{n-|V_S|}{n}
    = \left(1-\frac{|V_S|}{n}\right)^{d_2(u,v)}.
  \end{aligned}\]
  In particular, if $|S|\ge\eps m$,
  we have $|V_S|\ge\eps n/2$,
  hence
  \[
    \pr(\cM_S(u,v))
    \le (1-\eps/2)^k\le\exp(-\eps k/2).
  \]
  Let $\cM(u,v)$ be the event that $\cM_S(u,v)$
  holds for some $S\subseteq[m]$ with $|S|=t:=\floor{\eps m}$.
  By the union bound,
  \[\begin{aligned}
    \pr\left(\cM(u,v)\right)
    &\le \binom{m}{t}\exp(-\eps k/2)\le \left(\frac{e m }{t}\right)^{t} \exp(-\eps k/2)
    \\
    &\le \exp\left(\eps m\log\left(\frac{3}{\eps}\right)-\eps k/2\right)
    \le \exp\left(\eps k\left(c\log\left(\frac{3}{\eps}\right)-\frac{1}{2}\right)\right).
  \end{aligned}\]

  By assumption,
  $c\log(3/\eps)\le 1/4$.
  Thus, we get $\pr(\cM(u,v))\le\exp(-\eps k/4)$.
  Let $\cM$ be the event that $\cM(u,v)$ holds for some $u,v\in V$.
  By the union bound over all such pairs $u,v$, we have
  \[
    \pr(\cM)
    \le \binom{n}{2}\exp(-\eps k/4)
    \le n^{2-C\eps /4}.
  \]
  Since $C>8/\eps$,
  we obtain $\pr(\cM)=o(1)$.
  The statement follows by noticing that $\neg\cM(u,v)$ implies
  $|M_\pi(u,v)|\ge(1-\eps)m$,
  hence $\neg{\cM}$ implies $\delta_2(\pi)\ge(1-\eps)m$.
\end{proof}

\begin{proof}[Proof of \cref{thm:codegree}]
  Apply \cref{lem:codegree}
  with $\eps=1/8$, $c=1/13$, and $C\ge 65$
  (one can check that $c<1/(4\log(24))$).
  Assume $k\ge C\log{n}$.
  By \cref{lem:codegree},
  for sufficiently large $n$,
  there exists a partition $\pi=(V_1,\dots,V_m)$ of $V$
  into $m=\floor{ck}$ parts
  for which
  $\delta_2(\pi)\ge (1-\eps)m$.
  In particular,
  the reduced graph
  $G_\pi=([m],E_\pi)$
  where $\{i,j\}\in E_\pi$
  if and only if
  $E_G(V_i,V_j)\ne\es$
  has $\delta(G_\pi)\ge\delta_2(G_\pi)\ge \delta_2(\pi)-2\ge (1-\eps)m-2=\frac{7}{8}m-2$.
  Hence, by \cref{thm:JLV},
  $G_\pi$ is $(\lfloor\alpha m\rfloor-2)$-rigid
  for $\alpha=1/2-\eps=3/8$.
  Write $d=\floor{k/40}$,
  and note that $\lfloor \alpha m\rfloor-2\ge d$
  (for sufficiently large $k$).
  Note also that since $\delta_2(\pi)\ge \frac{7}{8}m> d$,
  for every $1\le i\le m$ and every pair of distinct vertices $u,v\in V_i$,
  there are at least $d$ distinct indices $j\in[m]$ such that $u$ and $v$
  are connected in $G[V_i,V_j]$.
  By \cref{thm:strong_rigid_partition_1}, $G$ is $d$-rigid.
\end{proof}

\begin{remark}
Let us note that instead of applying \cref{thm:strong_rigid_partition_1} in the proof of \cref{thm:codegree}, it is possible to obtain the same conclusion by a repeated application of Whiteley's vertex splitting lemma (\cref{lemma:whiteley_splitting_2}).
\end{remark}

\section{Generalized rigid partitions}\label{sec:partitions}

In this section, we present a new sufficient condition for $d$-rigidity, proving  \cref{thm:strong_rigid_partition_1} as a special case.

As a warm-up, let us first present a short proof of a slightly weaker version of \cref{thm:strong_rigid_partition_1} (which, in fact, is sufficient for all the applications of \cref{thm:strong_rigid_partition_1} presented in this paper).

\begin{theorem}\label{prop:strong_type2_rigid_partition}
       Let $d\ge 1$ and $m\ge d+1$. Let $G=(V,E)$ be a graph, and let $V_1,\ldots, V_m$ be a partition of $V$.
   Assume that the reduced graph $G'$ is $d$-rigid,
   and that for every $1\le i\le m$ and every pair of distinct vertices $u,v\in V_i$, there are at least $d$ distinct indices $j\in[m]\sm\{i\}$ such that $u$ and $v$ lie in the same connected component of $G[V_i,V_j]$. Then, $G$ is $d$-rigid.
\end{theorem}
Note that the only difference between \cref{prop:strong_type2_rigid_partition} and \cref{thm:strong_rigid_partition_1} is that in \cref{prop:strong_type2_rigid_partition} we ignore the edges inside each part $V_i$. Before the proof, let us recall a few basic facts about infinitesimal rigidity.

Let $G=(V,E)$ and $\p:V\to\mathbb{R}^d$. An \emph{infinitesimal motion} of $(G,\p)$ is a map $\bq:V\to\mathbb{R}^d$ such that $(\bq(u)-\bq(v))\cdot (\p(u)-\p(v))=0$ for all $\{u,v\}\in E$. An infinitesimal motion $\bq$ is called \emph{trivial} if there exist a skew-symmetric matrix $A\in\mathbb{R}^{d\times d}$ and a vector $x\in\mathbb{R}^d$ such that $\bq(v)=A \p(v)+x$ for all $v\in V$. The framework $(G,\p)$ is called \emph{infinitesimally rigid} if every infinitesimal motion of $(G,p)$ is trivial.

The rigidity matrix of $(G,\p)$, denoted by $R(G,\p)$, is the $|E|\times d|V|$ matrix whose rows are indexed by edges of $G$, and columns indexed by vertices (with $d$ consecutive columns associated with each vertex), defined by
\[
    R(G,\p)_{e,u}=\begin{cases}
       \p(u)-\p(v) & \text{if } e=\{u,v\} \text{ for some } v\in V,\\
        0 & \text{otherwise,}
    \end{cases}
\]
for all $e\in E$ and $u\in V$. Note that the infinitesimal motions (considered as vectors in $\mathbb{R}^{d|V|}$) are exactly the solutions to the system of equations $R(G,\p)\bq=0$.

It is a well-known fact (see~\cites{AR78, AR79}) that, assuming that the affine span of $\p(V)$ is at least $(d-1)$-dimensional, the space of trivial infinitesimal motions of $(G,\p)$ has dimension $\binom{d+1}{2}$. Therefore, in this case, $(G,\p)$ is infinitesimally rigid if and only if $\text{rank}(R(G,\p))= d|V|-\binom{d+1}{2}$.

It is easy to show that if $(G,\p')$ is infinitesimally rigid for some $\p':V\to \RR^d$, then $(G,\p)$ is infinitesimally rigid for every generic $\p:V\to\RR^d$.
Moreover, every infinitesimally rigid framework $(G,\p)$ is rigid, and, for generic $\p$, the notions of rigidity and infinitesimal rigidity are equivalent. Hence, a graph $G$ is $d$-rigid if and only if there exists $\p:V\to \RR^d$ such that $(G,\p)$ is infinitesimally rigid (see~\cite{AR79}).

\begin{proof}[Proof of \cref{prop:strong_type2_rigid_partition}]
Let $\p':[m]\to\mathbb{R}^d$ be a generic embedding. Since $G'$ is $d$-rigid, $(G',\p')$ is infinitesimally rigid.
We extend $\p'$ to an embedding $\p$ of $V$ by setting $\p(v)=\p'(i)$ whenever $v\in V_i$.
We will show that $(G,\p)$ is infinitesimally rigid, and therefore $G$ is $d$-rigid.

Assume for contradiction that $(G,\p)$ is not infinitesimally rigid.
Then, there exists a non-trivial infinitesimal motion $\bq:V\to\mathbb{R}^d$ of $(G,\p)$.

First, we will show that for every $i\in[m]$ and $u,v\in V_i$, we must have $\bq(u)=\bq(v)$. Let $i\in[m]$ and $u,v\in V_i$. Let $j_1,\ldots,j_d\in[m]\sm\{i\}$ for which $u$ and $v$ are connected by a path in $G[V_i,V_{j}]$ for all $j\in\{j_1,\ldots,j_d\}$. Let $1\le k\le d$. Note that for every $x,y\in V_i\cup V_{j_k}$, if $x$ and $y$ are adjacent in $G[V_i,V_{j_k}]$, then $(\bq(x)-\bq(y))\cdot (\p'(i)-\p'(j_k))=0$. By transitivity, we obtain $\bq(u)\cdot (\p'(i)-\p'(j_k))= \bq(v)\cdot (\p'(i)-\p'(j_k))$.
Since $\p'(i)-\p'(j_1),\ldots,\p'(i)-\p'(j_d)$ are linearly independent (as $\p'$ is generic), we must have $\bq(u)=\bq(v)$.

Now, we define $\bq':[m]\to \mathbb{R}^d$ by $\bq'(i)=\bq(u)$, where $u$ is a vertex in $V_i$, for all $1\le i\le m$ (note that $\bq'$ is well-defined by our previous argument).  Note that $\bq'$ is an infinitesimal motion of $(G',\p')$. Indeed, for every $\{i,j\}\in E'$, there exist $u\in V_i$ and $v\in V_j$ for which $\{u,v\}\in E$, and so
\[
    (\bq'(i)-\bq'(j))\cdot (\p'(i)-\p'(j))= (\bq(u)-\bq(v))\cdot (\p(u)-\p(v))=0,
\]
since $\bq$ is an infinitesimal motion of $(G,\p)$.
Moreover, since $\bq$ is non-trivial, $\bq'$ is non-trivial as well (indeed, if $\bq'(i)=A \p'(i)+z$ for all $i\in [m]$, for some skew-symmetric $A\in \mathbb{R}^d$ and $z\in \mathbb{R}^d$, then $\bq(v)= A \p(v)+z$ for all $v\in V$, in contradiction to $\bq$ being non-trivial). But this is a contradiction to $(G',\p')$ being infinitesimally rigid.
\end{proof}

In order to define a more general sufficient condition for rigidity, we first need to introduce the notion of \emph{anchored graphs}, presented next.

\begin{definition}
    Let $d\ge 1$. Let $G=(V,E)$ be a multi-graph, and  let $c:E\to \mathbb{N}$ be a map assigning a ``colour'' to each edge. Let $\{x_{i}\}_{i\in c(E)}\subseteq\RR^d$ be a generic set of points.
    A map $\bq:V\to\RR^d$ is called a \emph{motion} of $(G,c)$ if
    \[
        (\bq(u)-\bq(v))\cdot x_{c(e)} =0
    \]
    for every $e\in E$ with endpoints $u$ and $v$.
    A motion $\bq:V\to\RR^d$ is called \emph{trivial} if $\bq(u)=\bq(v)$ for all $u,v\in V$.
    We say that $(G,c)$ is \emph{$d$-anchored} if every motion of $(G,c)$ is trivial.
\end{definition}

Next, we introduce the notion of \emph{generalized rigid partitions}.

\begin{definition}\label{def:partitions}
Let $G=(V,E)$ be a graph and let $m\ge d\ge 1$. Let $V_1,\ldots,V_m$ be a partition of $V$. For all $i\in[m]$, let $m_i\ge m$, and let $\mathcal{I}=\{(i,j):\, 1\le i\le m,\, i<j\le m_i\}$. Let $\{G_{ij}\}_{(i,j)\in\mathcal{I}}$ be a family of pairwise edge-disjoint subgraphs of $G$ such that $V(G_{ij})= V_i\cup V_j$ for all $1\le i< j\le m$, and $V(G_{ij})=V_i$ for $1\le i\le m$ and $m<j\le m_i$. Let $\hat{G}=\bigcup_{(i,j)\in\mathcal{I}} G_{ij}\subseteq G$. For convenience, for $1\le j<i\le m$, we denote $G_{ij}=G_{ji}$. The \emph{reduced graph} associated with $(G,\{V_i\}_{i=1}^m, \{G_{ij}\}_{(i,j)\in\mathcal{I}})$ is the graph $G'=([m],E')$, where
 \[
      E'=\left\{\{i,j\}\in\binom{[m]}{2}:\, E_{G_{ij}}(V_i,V_j)\ne \es\right\}.
\]
For every $1\le i\le m$, the $i$-th \emph{anchoring graph} of $(G,\{V_i\}_{i=1}^m, \{G_{ij}\}_{(i,j)\in\mathcal{I}})$ is the edge-coloured multi-graph $(H_i,c_i)$ on vertex set $V_i$, having, for every $j\in[m_i]\sm\{i\}$, an edge $\{u,v\}$ with $c_i(\{u,v\})=j$ if $u$ and $v$ are connected by a path in $G_{ij}$.

Assume that the following three properties hold.
\begin{description}[leftmargin=*]
    \item[Monochromatic cuts property:] For every $i\in [m]$ and $U\subseteq V_i$ with $|U|\ge 2$, there exist $(i,j)\in\mathcal{I}$ and a partition $U=U'\cup U''$ such that $E_{\hat{G}}(U',U'')\subseteq E(G_{ij})$.
    \item[Anchored parts property:] For every $i\in [m]$, the $i$-th anchoring graph $(H_i,c_i)$ is $d$-anchored.
    \item[Rigid reduced graph:] The reduced graph $G'$ is $d$-rigid.
\end{description}
 Then, we say that $(\{V_i\}_{i=1}^m, \{G_{ij}\}_{(i,j)\in\mathcal{I}})$ is a \emph{generalized $d$-rigid partition} of $G$.
\end{definition}

Our main result in this section is the following theorem.

\begin{theorem}\label{thm:gen_partitions}
    If $G$ admits a generalized $d$-rigid partition, then $G$ is $d$-rigid.
\end{theorem}

\Cref{thm:gen_partitions} reduces, in the special case $m\le d+1$, to~\cite[Theorem 1.1]{KLM25}, which in turn generalizes results of Crapo~\cite{crapo1990plane} and Lindemann~\cite{LinPhD} corresponding to the $d=2,\, m\le 3$ and $d=3,\, m\le 4$ cases respectively. Moreover, as we will see next, it implies \cref{thm:strong_rigid_partition_1}, which, as we mentioned in the introduction, extends results from~\cite{LNPR23+} and~\cite{KLM25}, as well as Whiteley's vertex splitting lemmas from~\cites{whiteley90vertex,whiteley96some}.

\subsection{Generalized strong rigid partitions}

The notion of generalized rigid partitions introduced in \cref{def:partitions} is relatively complex and hard to apply directly. Therefore, we present next a notion of \emph{generalized strong rigid partitions}, which provides a more concrete sufficient condition for rigidity, and extends the notions of ``strong rigid partitions'' introduced in~\cite{LNPR23+} and~\cite{KLM25}.

\begin{definition}
Let $G=(V,E)$ be a graph. Let $m\ge d\ge 1$, and let $V_1,\ldots,V_m$ be a partition of $V$.  Recall that the \emph{reduced graph} of $(G,\{V_i\}_{i=1}^m)$ is defined as $G'=\{[m],E'\}$, where
 \[
      E'=\left\{\{i,j\}\in\binom{[m]}{2}:\, E(V_i,V_j)\ne \es\right\}.
\]
Assume that $G'$ is $d$-rigid. For $1\le i\le m$, let $Q_i$ be the graph on vertex set $V_i$ whose edges are the pairs $\{u,v\}\subseteq V_i$ for which there exist $d$ distinct indices $j_1,\ldots,j_d\in [m]$ such that $u$ and $v$ are connected by a path in $G[V_i,V_{j_k}]$ for all $k=1,\ldots,d$. Assume that $Q_i$ is connected for each $1\le i\le m$.
Then, we say that $V_1,\ldots,V_m$ is a \emph{generalized strong $d$-rigid partition} of $G$.
\end{definition}

\begin{theorem}\label{thm:strong_rigid_partition_2}
    If $G$ admits a generalized strong $d$-rigid partition, then $G$ is $d$-rigid.
\end{theorem}

Note that \cref{thm:strong_rigid_partition_2} immediately implies \cref{thm:strong_rigid_partition_1} (which requires the graphs $Q_i$ to be complete, instead of just connected). For the proof of \cref{thm:strong_rigid_partition_2}, we need the following very simple lemma.

\begin{lemma}\label{lemma:trivial_anchoring}
    Let $d\ge 1$. Let $G=(V,E)$ be a multi-graph,  let $c:E\to \mathbb{N}$, and let $\{x_i\}_{i\in c(E)}$ be a generic set of points in $\RR^d$. Assume that $\bq:V\to\RR^d$ is a motion of $(G,c)$, and let $u,v\in V$. If $u$ and $v$ are connected by a monochromatic path of colour $i$, then $(\bq(u)-\bq(v))\cdot x_i=0$.
\end{lemma}
\begin{proof}
    Since $(\bq(w)-\bq(z))\cdot x_i=0$ for every pair of vertices $w,z$ forming an $i$-coloured edge, the claim follows trivially by transitivity.
\end{proof}

\begin{proof}[Proof of \cref{thm:strong_rigid_partition_2}]
    Define, for $i\ne j$, $G_{ij}=G[V_i,V_j]$. In addition, define for all $1\le i\le m$, $G_{i,m+1}=G[V_i]=G[V_i,V_i]$. Let $\mathcal{I}=\{(i,j):\, 1\le i\le m:\, i<j\le m+1\}$.
    We will show that $(\{V_i\}_{i=1}^m,\{G_{ij}\}_{(i,j)\in \mathcal{I}})$ is a generalized $d$-rigid partition of $G$. Indeed, the reduced graph $G'$ is $d$-rigid by assumption, and the monochromatic cuts condition holds trivially, since each $V_i$ contains edges of just one of the subgraphs $G_{ij}$ (namely, of the graph $G_{i,m+1}=G[V_i]$). We are left to show that the anchored parts property holds.

    Let $x_1,\ldots,x_{m}$ be a generic set of points. Let $1\le i\le m$, and let $\bq:V_i\to\RR^d$ be a motion of the $i$-th anchoring graph $(H_i,c_i)$ associated with $(G,\{V_i\}_{i=1}^m,\{G_{ij}\}_{(i,j)\in \mathcal{I}})$. We need to show that $\bq(u)=\bq(v)$ for all $u,v\in V_i$. Since $Q_i$ is connected, it is enough to show this for $u$ and $v$ that are adjacent in $Q_i$. Let $u,v\in V_i$ such that $\{u,v\}\in E(Q_i)$. By the definition of $Q_i$, there exist indices $j_1,\ldots,j_d\in[m]$ such that $u$ and $v$ are connected by a path in $G[V_i,V_{j_k}]$ for all $k=1,\ldots,d$. By \cref{lemma:trivial_anchoring}, we have
    \[
        (\bq(u)-\bq(v))\cdot x_{j_k} = 0
    \]
    for $k=1,\ldots, d$. Since the points $x_j$ are generic, $x_{j_1},\ldots,x_{j_d}$ are linearly independent, and therefore we must have $\bq(u)-\bq(v)=0$. That is, $\bq(u)=\bq(v)$. Hence, $\bq$ is trivial, and therefore $(H_i,c_i)$ is $d$-anchored, as wanted. Hence, $(\{V_i\}_{i=1}^m,\{G_{ij}\}_{(i,j)\in \mathcal{I}})$ is a generalized $d$-rigid partition of $G$, and so, by \cref{thm:gen_partitions}, $G$ is $d$-rigid.
\end{proof}

\subsection{Limit frameworks}

For the proof of Theorem \ref{thm:gen_partitions}, we need to introduce the notion of limit frameworks and to recall some relevant facts about them.

Let $G=(V,E)$ be a graph, let $\p:V\to\mathbb{R}^d$ and $g: \{(u,e):\, e\in E,\, u\in e\}\to \mathbb{R}^d$. We say that $(G,\p,g)$ is a \emph{$d$-dimensional limit framework} if there exists a sequence $\{\p_n: V\to \mathbb{R}^d\}_{n=1}^{\infty}$ such that, for every $v\in V$,
\[
    \lim_{n\to\infty} \p_n(v)= \p(v),
\]
and for every $e\in E$ and $u\in e$,
\[
    \lim_{n\to\infty} \frac{\p_n(u)-\p_n(v)}{\|\p_n(u)-\p_n(v)\|}= g(u,e),
\]
where $v$ is the unique vertex in $e\sm\{u\}$. In such case, we say that the sequence $\{(G,\p_n)\}_{n=1}^{\infty}$ \emph{converges} to $(G,\p,g)$.

Note that if $(G,\p,g)$ is a limit framework then $\|g(u,e)\|=1$ for all $e\in E$ and $u\in e$, and that $g(u,e)=-g(v,e)$ for all $e=\{u,v\}\in E$.

We say that a map $\bq:V\to \mathbb{R}^d$ is an \emph{infinitesimal motion} of $(G,\p,g)$ if, for every $e=\{u,v\}\in E$,
\[
    (\bq(u)-\bq(v))\cdot g(u,e)=0.
\]

We say that $\bq$ is \emph{trivial} if there exist a skew-symmetric matrix $A\in\mathbb{R}^{d\times d}$ and a vector $x\in \mathbb{R}^d$ such that $\bq(v)=A \p(v)+x$ for all $v\in V$.  A limit framework $(G,\p,g)$ is called \emph{infinitesimally rigid} if every infinitesimal motion of $(G,\p,g)$ is trivial.

We define the rigidity matrix $R(G,\p,g)$ of a limit framework $(G,\p,g)$ as the $|E|\times d|V|$ matrix whose rows are indexed by edges of $G$, and columns indexed by vertices (with $d$ consecutive columns associated with each vertex), with entries given by
\[
    R(G,\p)_{e,u}=\begin{cases}
        g(u,e) & \text{if } u\in e,\\
        0 & \text{otherwise,}
    \end{cases}
\]
for all $e\in E$ and $u\in V$. Similarly to the non-limiting case, it is easy to check that, assuming that the affine span of $\p(V)$ is at least $(d-1)$-dimensional, $(G,\p,g)$ is infinitesimally rigid if and only if $\text{rank}(R(G,\p,g))= d|V|-\binom{d+1}{2}$.

\begin{lemma}\label{lemma:limit_rigidity}
    Let $G=(V,E)$ be a graph, and let $(G,\p,g)$ be a $d$-dimensional limit framework of $G$. Assume that the affine span of $\p(V)$ is at least $(d-1)$-dimensional. If $(G,\p,g)$ is infinitesimally rigid, then $G$ is $d$-rigid.
\end{lemma}
\begin{proof}
    Let $\{\p_n:V\to \mathbb{R}^d\}_{n=1}^{\infty}$ be a sequence such that $\{(G,\p_n)\}_{n=1}^{\infty}$ converges to $(G,\p,g)$. We will show that for $n$ large enough, $(G,\p_n)$ is infinitesimally rigid, and therefore $G$ is $d$-rigid.  Let $m=\dim(\p(V))$ denote the dimension of the affine span of the image of $\p$. Similarly, for $n\ge 1$, let $m_n=\dim(\p_n(V))$. Note that, for large enough $n$, $m_n\ge m\ge d-1$.

    For $n\ge 1$, let $R_n$ be the matrix obtained from $R(G,\p_n)$ by dividing, for each edge $e=\{u,v\}\in E$ such that $\p_n(u)\ne \p_n(v)$, the row associated with $e$ by $\|\p_n(u)-\p_n(v)\|$. Note that $R_n$ has the same rank as $R(G,\p_n)$, and that the sequence $R_n$ converges entry-wise to $R(G,\p,g)$. Therefore, there exists $n\ge 1$ such that $m_n\ge m \ge d-1$ and
  \[
        \text{rank}(R(G,\p_n))=\text{rank}(R_n)\ge \text{rank}(R(G,\p,g)) = d|V|-\binom{d+1}{2},
   \]
    where we used the fact that $(G,\p,g)$ is infinitesimally rigid. Therefore, $(G,\p_n)$ is infinitesimally rigid, as wanted.
\end{proof}

For the proof of \cref{thm:gen_partitions}, we will need the following auxiliary results on limit frameworks, which were proved in~\cite{KLM25} (and which rely on and extend earlier work by Tay~\cite{Tay1993laman}).

\begin{lemma}[{\cite[Lemma 2.4]{KLM25}}]
\label{lemma:constructing_limit_frameworks}
Let $G=(V,E)$. Let $V=V_1\cup\cdots\cup V_m$ be a partition of $V$, and let $x_1,\ldots,x_m$ be $m$ distinct points in $\RR^d$.
Assume we have for each $1\leq i\leq m$ a $d$-dimensional limit framework $(G[V_i],\p_i,g_i)$.
Define $g: \{(u,e):\, e\in E, u\in e\}\to \RR^d$ by
\[
g(u,\{u,v\})= \begin{cases}
    (x_i-x_j)/\|x_i-x_j\| & \text{ if } u\in V_i, v\in V_j \text{ for } i\neq j,\\
    g_i(u,\{u,v\}) & \text{ if } u,v\in V_i \text{ for } 1\leq i\leq m,
\end{cases}
\]
for all $\{u,v\}\in E$. Let $\p:V\to \RR^d$ be defined by $\p(v)=x_i$ for each $i\in[m]$ and $v\in V_i$. Then, $(G,\p,g)$ is a $d$-dimensional limit framework.
\end{lemma}

\begin{lemma}[{\cite[Lemma 2.5]{KLM25}}]
\label{lemma:combing}
  Let $G=(V,E)$ be a graph, and let $E=E_1\cup \cdots\cup E_k$ be a partition of its edge set.
  Let $y_1,\ldots, y_k\in \RR^d$ with $\|y_i\|=1$ for all $1\le i\le k$.
  If for every $U\subseteq V$ of size at least $2$  there exist $1\le i\le k$ and a partition $U',U''$ of $U$ such that $E(U',U'')\subseteq E_i$,
  then there exists a $d$-dimensional limit framework $(G,\p,g)$ such that, for every $i\in [k]$,
  $g(u,\{u,v\})\in \{y_i,-y_i\}$ for all $\{u,v\}\in E_i$.
\end{lemma}

\subsection{Proof of \cref{thm:gen_partitions}}

We proceed to prove \cref{thm:gen_partitions}. The proof outline is as follows: First, we show, using the monochromatic cuts property, that there is a $d$-dimensional limit framework $(G,\p,g)$ with certain nice properties (in particular, such that the vertices in each part $V_i$ are mapped to a single point $x_i$). Then, we apply the anchored parts property to show that in every infinitesimal motion $\bq$ of $(G,\p,g)$, the vertices in each $V_i$ must obtain the same value. This allows us to reduce $\bq$ to an infinitesimal motion on the reduced graph $G'$. Since $G'$ is assumed to be $d$-rigid, this implies that $\bq$ is a trivial infinitesimal motion, showing that $G$ is $d$-rigid as well.

\begin{proof}[Proof of \cref{thm:gen_partitions}]
    Let $G=(V,E)$ be a graph, let $m\ge d\ge 1$. Let $m_i\ge m$ for all $i\in[m]$ and $\mathcal{I}=\{(i,j):\, 1\le i\le m,\, i<j\le m_i\}$. Let $(\{V_i\}_{i=1}^m, \{G_{ij}\}_{(i,j)\in\mathcal{I}})$ be a generalized $d$-rigid partition of $G$. Let $m'=\max\{m_i:\, 1\le i\le m\}$, and let $x_1,\ldots,x_{m'}\in \RR^d$ be a generic set of points.  For $i\ne j$, let $y_{ij}= (x_i-x_j)/\|x_i-x_j\|$.

    Since addition of edges preserves $d$-rigidity, we may assume without loss of generality that $E=\bigcup_{(i,j)\in\mathcal{I}} E(G_{ij})$. By Lemma \ref{lemma:combing} (and using the monochromatic cuts property), for every $1\le i\le m$ there exists a $d$-dimensional limit framework $(G[V_i],\p_i,g_i)$ such that $g_i(u,e)\in\{y_{ij},-y_{ij}\}$ for every $j\in[m_i]\sm\{i\}$, $e\in E(G_{ij}[V_i])$ and $u\in e$. Hence, by Lemma \ref{lemma:constructing_limit_frameworks}, there exists a $d$-dimensional limit framework $(G,\p,g)$ such that $\p(v)=x_i$ for every $i\in[m]$ and $v\in V_i$, and
    $g(u,e)\in\{y_{ij},-y_{ij}\}$ for every $(i,j)\in\mathcal{I}$, $e\in E(G_{ij})$ and $u\in e$.

    Let $G'=([m],E')$ be the reduced graph of $(G,\{V_i\}_{i=1}^m, \{G_{ij}\}_{(i,j)\in\mathcal{I}})$. Let $\p':[m]\to\RR^d$ be defined by $\p'(i)=x_i$ for all $i\in[m]$.
    Since $G'$ is $d$-rigid and $\p'$ is generic, the framework $(G',\p')$ is infinitesimally rigid.

    We will show that $(G,\p,g)$ is infinitesimally rigid, and therefore, by Lemma \ref{lemma:limit_rigidity}, $G$ is $d$-rigid.
    Let $\bq:V\to \RR^d$ be an infinitesimal motion of $(G,\p,g)$.
    Note that, for every $(i,j)\in\mathcal{I}$ and $e=\{u,v\}\in E(G_{ij})$, we have $g(u,e)=-g(v,e)\in \{y_{ij},-y_{ij}\}$, and therefore
    \[
        (\bq(u)-\bq(v))\cdot y_{ij}= (\bq(u)-\bq(v))\cdot g(u,e)=0.
    \]
    That is, $(\bq(u)-\bq(v))\cdot (x_i-x_j)= 0$ whenever $u$ and $v$ are adjacent in $G_{ij}$, or more generally whenever they are connected by a path in $G_{ij}$. Since the set $x_1,\ldots,x_{m'}$ is generic, this equation holds also under the specialization $x_i=0$. Therefore,
    \[
       (\bq(u)-\bq(v))\cdot x_j = 0
    \]
    for all $u,v\in V_i\cup V_j$ that are connected by a path in $G_{ij}$.  In particular, this shows that, for all $1\le i\le m$, the restriction of $\bq$ to $V_i$ is a motion of the $i$-th anchoring graph $(H_i,c_i)$. By the anchored parts condition, $(H_i,c_i)$ is $d$-anchored for all $1\le i\le m$, and therefore $\bq(u)=\bq(v)$ for all $1\le i\le m$ and $u,v\in V_i$.

    Now, define $\bq':[m]\to \RR^d$ by $\bq'(i) = \bq(u)$ for all $i\in [m]$, where $u$ is some vertex in $V_i$. Note that this is well defined since $\bq(u)=\bq(v)$ for all $u,v\in V_i$.
    Let $1\le i< j\le m$ with $\{i,j\}\in E'$. By the definition of $G'$, there exist $u\in V_i$ and $v\in V_j$ such that $\{u,v\}\in E(G_{ij})$.  Since $\bq$ is an infinitesimal motion of $(G,\p,g)$, we obtain
    \[
        (\bq'(i)-\bq'(j))\cdot \frac{x_i-x_j}{\|x_i-x_j\|} = (\bq(u)-\bq(v))\cdot y_{ij}= (\bq(u)-\bq(v))\cdot g(u,\{u,v\})=0.
    \]
    Therefore, $\bq'$ is an infinitesimal motion of $(G',\p')$. Since $(G',\p')$ is infinitesimally rigid, $\bq'$ must be trivial. That is, there are a  skew-symmetric matrix $A\in\RR^{d\times d}$ and vector $x\in \RR^d$ such that $\bq'(i)= A x_i + x$ for all $i\in[m]$. But then, for all $i\in [m]$ and $v\in V_i$, we obtain
    \[
        \bq(v) = \bq'(i)= A x_i +x = A\p(v)+x.
    \]
    That is, $\bq$ is a trivial infinitesimal motion of $(G,\p,g)$. Thus, $(G,\p,g)$ is infinitesimally rigid, as wanted.
\end{proof}

\section{Concluding remarks}\label{sec:concluding}

\subsection{Another variant of rigid partitions}

Let us mention the following variant of \cref{thm:strong_rigid_partition_2}, which demonstrates the additional strength of \cref{thm:gen_partitions}, and may be useful, for example, in the study of rigidity of bipartite graphs.

\begin{theorem}\label{thm:double_partition}
    Let $m\ge d\ge 1$.  Let $G=(V,E)$ be a graph, and let $V_1,\ldots,V_m$ be a partition of $V$. For each $i\in[m]$, let $k_i\ge 1$, and let $V_{i1},\ldots,V_{i k_i}$ be a partition of $V_i$. Assume that the reduced graph $G'$ associated with $(G,\{V_1,\ldots,V_m\})$ is $d$-rigid. Moreover, assume that for all $i\in [m]$, $G[V_i]$ contains a forest $F_i$ such that the multi-graph obtained from $F_i$ by contracting each set $V_{ij}$, for $1\le j\le k_i$, into a single vertex (and removing self-loops), contains $d$ edge-disjoint spanning trees. Finally, for every $i\in [m]$ and $j\in [k_i]$, let $Q_{ij}=(V_{ij},E'_{ij})$, where $E'_{ij}$ consists of all pairs of vertices $u,v\in V_{ij}$ for which there are at least $d$ distinct indices $s\in[m]\sm\{i\}$ such that $u$ and $v$ are connected by a path in $G[V_i,V_s]$. Assume that $Q_{ij}$ is connected for all $i\in[m]$ and $j\in[k_i]$. Then, $G$ is $d$-rigid.
\end{theorem}
\begin{proof}

    For $1\le i\le m$, let $m_i=m+|E(F_i)|$, and let $\mathcal{I}=\{(i,j):\, 1\le i\le m,\, i<j\le m_i\}$.
    For $1\le i<j\le m$, let $G_{ij}=G_{ji}=G[V_i,V_j]$. For $1\le i\le m$, enumerate the edges of $F_i$ arbitrarily as $e_1,\ldots,e_{|E(F_i)|}$. For $m+1\le j\le m_i$, let $G_{ij}$ be the graph on vertex set $V_i$ whose edge set consists of the single edge $e_{j-m}$.

    We will show that $(G,\{V_i\}_{i=1}^m,\{E_{ij}\}_{(i,j)\in\mathcal{I}})$ is a generalized $d$-rigid partition of $G$.
    Indeed, let us first show that the monochromatic cuts property is satisfied. Let $i\in[m]$ and $U\subseteq V_i$ with $|U|\ge 2$. The case when $U$ has no edges in any of the subgraphs $G_{ij}$ is trivial. Therefore, let us assume that there exists $e\in F_i$ such that $e\subseteq U$. Note that $F_i[U]$ is a non-empty forest on at least two vertices, and so it must have a leaf $u\in U$. Taking $U'=\{u\}$ and $U''=U\sm \{u\}$, we obtain a cut of $U$ with a single edge (and therefore, trivially, a monochromatic cut).

    Next, let us show that our partition satisfies the anchored parts condition. For $i\in[m]$, let $(H_i,c_i)$ be the $i$-th anchoring graph associated with our partition of $G$. Let $x_1,\ldots,x_{m_i}$ be a generic set of points in $\RR^d$. Let $\bq:V_i\to\RR^d$ be a motion of $(H_i,c_i)$. That is, for every $u,v\in V_i$ that are connected by a path in $G_{ij}$ for $j\in[m_i]\sm\{i\}$, we have
    \[
        (\bq(u)-\bq(v))\cdot x_{j}=0.
    \]
    First, note that, by \cref{lemma:trivial_anchoring} and the fact that $Q_{ij}$ is connected for every $j\in [k_i]$, we must have $\bq(u)=\bq(v)$ for every $j\in [k_i]$ and $u,v\in V_{ij}$.

    Let $F_i'$ be the coloured multi-graph obtained from $F_i$ by contracting each set $V_{ij}$ into a single vertex. Note that each edge of $F_i'$ has a distinct colour.
    Let $\bq': V(F'_i)\to\RR^d$ be the map induced by $\bq$. It is easy to check that $\bq'$ is a motion of $F_i'$.
    It follows from~\cite[Theorem 1]{whiteley1988union} that, since $F_i'$ contains $d$ edge-disjoint spanning trees, then $F_i'$ is $d$-anchored, and in particular $\bq'$ is trivial.
    Therefore, $\bq$ is trivial, as wanted. That is, $(H_i,c_i)$ is $d$-anchored.

    Finally, the reduced graph of $G$ associated with $(\{V_i\}_{i=1}^m,\{E_{ij}\}_{(i,j)\in\mathcal{I}})$ is $d$-rigid by assumption. Thus, $(G,\{V_i\}_{i=1}^m,\{E_{ij}\}_{(i,j)\in\mathcal{I}})$ is a generalized $d$-rigid partition of $G$, and so, by \cref{thm:gen_partitions}, $G$ is $d$-rigid.
\end{proof}

\Cref{thm:double_partition} generalizes~\cite[Lemma 2.11]{KLM25}, corresponding to the special case when $G=(A,B,E)$ is bipartite, $m=d+1$, and each part $V_i$ in the partition is divided into two parts, $V_i\cap A$ and $V_i\cap B$.

\subsection{Characterization of $d$-anchored graphs}
The notion of $d$-anchored coloured multi-graphs may be seen as a generalization of the rigidity of $d$-frames, studied by Whiteley in~\cite{whiteley1988union}, which corresponds to the special case when every edge in the multi-graph $G$ is assigned a different colour by the map $c$. As mentioned in the proof of \cref{thm:double_partition}, Whiteley showed that, in this case, the graph $(G,c)$ is $d$-anchored if and only if it contains $d$ edge-disjoint spanning trees. It would be interesting to obtain an extension of this characterization to the general case.
Note that a minimally $d$-anchored graph requires, in addition to being the union of $d$ edge-disjoint spanning trees, that, for every colour $k$, its $k$-coloured subgraph forms a forest, as otherwise the matrix associated with the $d$-anchoring system of equations for the graph will not be of full rank. However, these two necessary conditions are not sufficient. Indeed, let $(G,c)$ be the multi-graph on vertex set $[3]$ with two red edges $\{1,2\},\{1,3\}$, one blue edge $\{2,3\}$, and one green edge $\{2,3\}$. It is easy to check that $G$ is the edge-disjoint union of two spanning trees, and that each monochromatic subgraph of $G$ is a forest, but $(G,c)$ is not $2$-anchored.

\subsection{Some directions for further research}

It is natural to ask whether the lower bound $k=\Omega(\log{n})$ on the minimum codegree of an $n$-vertex graph $G$ is required for the conclusion of \cref{thm:codegree} to hold. Namely, is there an absolute constant $c>0$, such that for every $k\ge 1$, every graph $G$ with minimum codegree at least $k$ is $\lfloor ck\rfloor $-rigid?

Another natural question, which we plan to address in future work, is whether the methods in this paper could be used to obtain new results on the rigidity of pseudorandom graphs.
Recall that an $n$-vertex $k$-regular graph is called an $(n,k,\lambda)$-graph if the absolute values of all its non-trivial adjacency eigenvalues are at most $\lambda$ (see, for example,~\cite{KS06} for more details). In~\cite{KLM25}*{Theorem 1.7}, we showed that there exists $C>1$, such that if $G$ is an $(n,k,\lambda)$ graph with $k\ge \max\{9d\lambda, Cd \log{d}\}$, then $G$ is $d$-rigid. Vill\'anyi's results from~\cite{Vil23+} imply that an $(n,k,\lambda)$-graph is $\lfloor(1-o(1))\sqrt{k-\lambda}\rfloor$-rigid (see~\cite{KLM25} for more details). It would be interesting to determine whether there exist absolute constants $c_1,c_2>0$ such that every $(n,k,\lambda)$-graph $G$ with $\lambda/k<c_1$ is $\lfloor c_2 k\rfloor$-rigid.

\bibliography{library}

\end{document}